\documentclass[12pt,draft]{amsart}
\usepackage{amsmath,amssymb,amsthm,bm}
\usepackage{graphicx,enumerate}
\usepackage{layout}

\theoremstyle{plain}
\newtheorem{theorem}{Theorem}[section]
\newtheorem*{theorem-nn}{Theorem}
\newtheorem{lemma}[theorem]{Lemma}
\newtheorem{proposition}[theorem]{Proposition}
\newtheorem{corollary}[theorem]{Corollary}

\theoremstyle{definition}
\newtheorem*{definition}{Definition}
\newtheorem{example}[theorem]{Example}
\newtheorem{remark}[theorem]{Remark}

\newcommand{\bs}{\mathbf{s}}\newcommand{\bt}{\mathbf{t}}
\newcommand{\ba}{\mathbf{a}}\newcommand{\bb}{\mathbf{b}}
\newcommand{\bx}{\mathbf{x}}\newcommand{\by}{\mathbf{y}}
\newcommand{\Gs}{G_\mathbf{s}}\newcommand{\Gt}{G_\mathbf{t}}
\newcommand{\Gst}{G_{\mathbf{s},\mathbf{t}}}
\newcommand{\Hst}{H_{\mathbf{s},\mathbf{t}}}

\newcommand{\opi}{\overline{\pi}}
\newcommand{\Spl}{\mathrm{Spl}}\newcommand{\Gal}{\mathrm{Gal}}
\newcommand{\D}{\mathcal{D}}\newcommand{\cS}{\mathcal{S}}
\newcommand{\V}{\mathcal{V}}\newcommand{\A}{\mathcal{A}}
\newcommand{\C}{\mathcal{C}}\newcommand{\R}{\mathcal{R}}
\newcommand{\RP}{\mathcal{RP}}

\numberwithin{equation}{section}

\title[Field intersection problem of quartic generic polynomials]
{On the field intersection problem of quartic generic polynomials 
via formal Tschirnhausen transformation}
\author{Akinari Hoshi and Katsuya Miyake}

\thanks{This work was partially supported by Grant-in-Aid for Scientific 
Research (C) 19540057 of Japan Society for the Promotion of Science and 
Rikkyo University Special Fund for Research.}



\begin{document}
\maketitle

\begin{abstract}
Let $k$ be a field of characteristic $\neq 2$. 
We give an answer to the field intersection problem of quartic generic polynomials over $k$ 
via formal Tschirnhausen transformation and multi-resolvent polynomials. 
\end{abstract}

\section{Introduction}\label{seIntro}

Let $k$ be a field of char $k\neq 2$ and $k(\bs)$ the rational function field over $k$ 
with $n$ indeterminates $\bs=(s_1,\ldots,s_n)$. 
Let $G$ be a finite group. 
A polynomial $f_\bs(X)\in k(\bs)[X]$ is called $k$-generic for $G$ if 
the Galois group of $f_\bs(X)$ over $k(\bs)$ is isomorphic to $G$ and every 
$G$-Galois extension $L/M$ over an arbitrary infinite field $M\supset k$ can be obtained as 
$L=\Spl_M f_\ba(X)$, the splitting field of $f_\ba(X)$ over $M$, for some 
$\ba=(a_1,\ldots,a_n)\in M^n$ (cf. \cite{DeM83}, \cite{Kem01}, \cite{JLY02}). 
Note that we always take an infinite field $M$ as a base field $M$, $M\supset k$, of a 
$G$-extension $L/M$. 
Examples of $k$-generic polynomials for $G$ are known for various pairs of $(k,G)$ 
(for example, see \cite{Kem94}, \cite{KM00}, \cite{JLY02}, \cite{Rik04}). 

Let $f_\bs^G(X)\in k(\bs)[X]$ be a $k$-generic polynomial for $G$. 
Kemper \cite{Kem01} showed that for a subgroup $H$ of $G$ every $H$-Galois extension over 
an infinite field $M\supset k$ is also given by a specialization of $f_\bs^G(X)$ 
as in the similar manner. 
The aim of this paper is to study the field intersection problem $\mathbf{Int}(f_\bs^G/M)$ 
of $f_\bs^G(X)$ over $M$: 
\begin{center}
$\mathbf{Int}(f_\bs^G/M)$ : 
for a field $M\supset k$ and $\ba,{\ba'}\in M^n$, determine the\\
\hspace*{1.7cm} intersection of $\Spl_M f_\ba^G(X)$ and $\Spl_M f_{\ba'}^G(X)$. 
\end{center}

It would be desired to give an answer to the problem within the base field $M$ by using 
the data $\ba,{\ba'}\in M^n$. 
As a special case, this problem includes the field isomorphism problem 
{\bf $\mathbf{Isom}(f_\bs^G/M)$} of $f_\bs^G(X)$ over $M$, i.e., 
for $\ba,{\ba'}\in M^n$ whether $\Spl_M f_\ba^G(X)$ and $\Spl_M f_{\ba'}^G(X)$ 
are isomorphic over $M$ or not. 
Since a $k$-generic polynomial covers all $H$-Galois extensions ($H\leq G$) over $M\supset k$ 
by specializing parameters, the problem {\bf $\mathbf{Isom}(f_\bs^G/M)$} arises naturally. 
Moreover we consider the following problem: 
\begin{center}
{\bf $\mathbf{Isom}^\infty(f_\bs^G/M)$} : for a given $\ba\in M^n$, are there 
infinitely many\\
\hspace*{3.4cm} ${\ba'}\in M^n$ such that $\Spl_M f_\ba^G(X)=\Spl_M f_{\ba'}^G(X)$\,? 
\end{center}

Let $\cS_n$ (resp. $\A_n$, $\D_n$, $\C_n$) be the symmetric (resp. the alternating, 
the dihedral, the cyclic) group of degree $n$ and $\V_4$ the Klein four group 
($\V_4\cong\C_2\times\C_2)$.  
In \cite{HM07} and \cite{HM}, we gave answers to $\mathbf{Int}(f_\bs^G/M)$ and to 
$\mathbf{Isom^\infty}(f_\bs^G/M)$ for cubic $k$-generic polynomials 
$f_s^{\C_3}(X)=X^3-sX^2-(s+3)X-1$ and $f_s^{\cS_3}(X)=X^3+sX+s$. 
In the present paper we investigate the problems $\mathbf{Int}(f_\bs^G/M)$ and 
$\mathbf{Isom^\infty}(f_\bs^G/M)$ for quartic generic polynomials $f_\bs^G(X)$ via formal 
Tschirnhausen transformation and multi-resolvent polynomials. 

For $G=\cS_4$, $\D_4$, $\C_4$, $\V_4$, we take the following $k$-generic polynomials 
\begin{align*}
f_{s,t}^{\cS_4}(X)&:=X^4+sX^2+tX+t\, \in k(s,t)[X],\\
f_{s,t}^{\D_4}(X)&:=X^4+sX^2+t\, \in k(s,t)[X],\\
f_{s,u}^{\C_4}(X)&:=X^4+sX^2+\frac{s^2}{u^2+4}\, \in k(s,u)[X],\\
f_{s,v}^{\V_4}(X)&:=X^4+sX^2+v^2\, \in k(s,v)[X],
\end{align*}
respectively, with two parameters (the least possible number of parameters; cf. \cite{BR97}, 
\cite[Chapter 8]{JLY02}). 

In Section \ref{sePre}, we review some known results about resolvent polynomials and
formal Tschirnhausen transformation. 

In Section \ref{seS4A4}, we give an answer to $\mathbf{Int}(f_{\bs}^{\cS_4}/M)$ via 
multi-resolvent polynomial (Theorem \ref{thS4A4}). 
In Subsection \ref{seIsoS4}, we give a more explicit answer to $\mathbf{Isom}(f_{\bs}^{\cS_4}/M)$ 
by using formal Tschirnhausen transformation in Theorem \ref{thS4}. 
A proof of Theorem \ref{thS4} will be given in Subsection \ref{seProof}. 
A consequence of Theorem \ref{thS4} is the following theorem: 
\begin{theorem-nn}[Corollary \ref{cor2}, an answer to $\mathbf{Isom}^\infty(f_{\bs}^{\cS_4}/M)$]
Let $M\supset k$ be an infinite field. 
For $\ba=(a,b)\in M^2$, we assume that $f_\ba^{\cS_4}(X)$ is separable over $M$. 
Then there exist infinitely many $\ba'=(a',b')\in M^2$ such that 
$\Spl_M f_\ba^{\cS_4}(X)=\Spl_M f_{\ba'}^{\cS_4}(X)$. 
\end{theorem-nn}
In Section \ref{seD4}, we treat the problems $\mathbf{Int}(f_\bs^{\D_4}/M)$, 
$\mathbf{Isom}(f_\bs^{\D_4}/M)$ and $\mathbf{Isom^\infty}(f_\bs^{\D_4}/M)$. 
In the case of $\D_4$, $\mathbf{Isom^\infty}(f_\bs^{\D_4}/M)$
has a trivial solution because $\Spl_M f_{a,b}^{\D_4}(X)=\Spl_M f_{ac^2,bc^4}^{\D_4}(X)$ 
for arbitrary $c\in M\backslash\{0\}$.
Thus we consider the problem $\mathbf{Isom^\infty}(f_\bs^{\D_4}/M)$ for $\ba=(a,b)$ and 
$\ba'=(a',b')$ under the condition $a^2b'-{a'}^2b\neq 0$ or $b'/b\neq c^4$ for any $c\in M$. 
\begin{theorem-nn}[Theorem \ref{thD4Hil}]
Let $M\supset k$ be a Hilbertian field. 
For $\ba=(a,b)\in M^2$, we assume that $f_\ba^{\D_4}(X)$ is separable over $M$. 
Then there exist infinitely many $\ba'=(a',b')\in M^2$ 
which satisfy that $b'/b$ is not a fourth power in $M$ and $\Spl_M f_\ba^{\D_4}(X)=\Spl_M f_{\ba'}^{\D_4}(X)$. 
\end{theorem-nn}
In Section \ref{seC4V4}, we deal with the cases of $\C_4$ and of $\V_4$ 
which are treated by suitably specializing the case of $\D_4$. 
We also treat reducible cases in Section \ref{seRed}. 

Most of results in the present paper are given with explicit formulas 
which are intended to be applied elsewhere, and we also give some numerical examples 
by using our explicit formulas. 
The calculations of this paper were carried out with Mathematica \cite{Wol03}.

\section{Preliminaries}\label{sePre}

In this section we review some basic facts, and a result of \cite{HM}. \\

\subsection{Resolvent polynomial}\label{subseResolv}
~\\

One of the fundamental tools in the computational aspects of Galois theory 
is the resolvent polynomials (cf. the text books \cite{Coh93}, \cite{Ade01}). 
Several kinds of methods to compute a resolvent polynomial have been developed by 
many mathematicians 
(see, for example, \cite{Sta73}, \cite{Gir83}, \cite{SM85}, \cite{Yok97}, \cite{MM97}, 
\cite{AV00}, \cite{GK00} and the references therein). 

Let $M\supset k$ be an infinite field and $\overline{M}$ a fixed algebraic closure of $M$.
Let $f(X):=\prod_{i=1}^m(X-\alpha_i) \in M[X]$ be a separable polynomial of degree $m$ with 
fixed ordering roots $\alpha_1,\ldots,\alpha_m\in \overline{M}$. 
The information of the splitting field $\Spl_M f(X)$ of $f(X)$ over $M$ 
and their Galois group is obtained by using resolvent polynomials. 

Let $k[\bx]:=k[x_1,\ldots,x_m]$ be the polynomial ring over $k$ with indeterminates 
$x_1,\ldots,x_m$. 
Put $R:=k[\bx, 1/\Delta_\bx]$, where $\Delta_\bx:=\prod_{1\leq i<j\leq m}(x_j-x_i)$. 
We take a surjective evaluation homomorphism 
\[
\omega_f : R \longrightarrow k(\alpha_1,\ldots,\alpha_m),\quad 
\Theta(x_1,\ldots,x_m)\longmapsto \Theta(\alpha_1,\ldots,\alpha_m)
\]
for $\Theta \in R$.
We note that $\omega_f(\Delta_\bx)\neq 0$ from the assumption that $f(X)$ is separable over $M$. 
The kernel of the map $\omega_f$ is the ideal 
\[
I_f=\mathrm{ker}(\omega_f)=\{\Theta(x_1,\ldots,x_m)\in R 
\mid \Theta(\alpha_1,\ldots,\alpha_m)=0\}. 
\]
For $\pi\in \cS_m$, we extend the action of $\pi$ on $m$ letters $\{1,\ldots,m\}$ 
to $R$ by 
\[
\pi(\Theta(x_1,\ldots,x_m)):=\Theta(x_{\pi(1)},\ldots,x_{\pi(m)}).
\]
We define the Galois group of a polynomial $f(X)\in M[X]$ over $M$ by 
\[
\Gal(f/M):=\{\pi\in \cS_m \mid \pi(I_f)\subseteq I_f\}.
\]
We write $\Gal(f):=\Gal(f/M)$ for simplicity. 
The Galois group of the splitting field $\Spl_M f(X)$ of a polynomial $f(X)$ over $M$ 
is isomorphic to $\Gal(f)$.  If we take another ordering of roots 
$\alpha_{\pi(1)},\ldots,\alpha_{\pi(m)}$ of $f(X)$ with some $\pi\in \cS_m$, 
the corresponding realization of $\Gal(f)$ is the conjugate of the original one given 
by $\pi$ in $\cS_m$. 
Hence, for arbitrary ordering of the roots of $f(X)$, $\Gal(f)$ 
is determined up to conjugation in $\cS_m$. 

\begin{definition}
For $H\leq G\leq \cS_m$, an element $\Theta\in R$ is called a $G$-primitive $H$-invariant if 
$H=\mathrm{Stab}_G(\Theta)$ $:=$ $\{\pi\in G\ |\ \pi(\Theta)=\Theta\}$. 
For a $G$-primitive $H$-invariant $\Theta$, the polynomial 
\[
\RP_{\Theta,G}(X):=\prod_{\opi\in G/H}(X-\pi(\Theta))\in R^G[X]
\]
is called the {\it formal} $G$-relative $H$-invariant resolvent by $\Theta$, and a polynomial
\[
\RP_{\Theta,G,f}(X):=\prod_{\opi\in G/H}\bigl(X-\omega_f(\pi(\Theta))\bigr)
\]
is called the $G$-relative $H$-invariant resolvent of $f$ by $\Theta$. 
\end{definition}

The following is fundamental in the theory of resolvent polynomials (cf. \cite[p.95]{Ade01}). 

\begin{theorem}\label{thfun}
For $H\leq G\leq \cS_m$, let $\Theta$ be a $G$-primitive $H$-invariant. 
Assume that $\Gal(f)\leq G$. 
Suppose that $\RP_{\Theta,G,f}(X)$ is decomposed into a product of powers of distinct 
irreducible polynomials as $\RP_{\Theta,G,f}(X)=\prod_{i=1}^l h_i^{e_i}(X)$ in $M[X]$. 
Then we have a bijection 
\begin{align*}
\Gal(f)\backslash G/H\quad &\longrightarrow \quad \{h_1^{e_1}(X),\ldots,h_l^{e_l}(X)\},\\
\Gal(f)\, \pi\, H\quad &\longmapsto\quad h_\pi(X)
=\prod_{\tau H\subseteq \Gal(f)\,\pi\,H}\bigl(X-\omega_{f}(\tau(\Theta))\bigr)
\end{align*}
where the product is taken over the left cosets $\tau H$ of $H$ in $G$ contained in 
$\Gal(f)\, \pi\, H$, that is, over $\tau=\pi_\sigma \pi$ where $\pi_\sigma$ runs 
through a system of representatives of the left cosets of $\Gal(f) \cap \pi H\pi^{-1}$ in $\Gal(f)$, and each 
$h_\pi(X)$ is irreducible or a power of an irreducible polynomial with $\mathrm{deg}(h_\pi(X))$ 
$=$ $|\Gal(f)\, \pi\, H|/|H|$ $=$ $|\Gal(f)|/|\Gal(f)\cap \pi H\pi^{-1}|$. 
\end{theorem}

\begin{corollary} 
If $\Gal(f)\leq \pi H\pi^{-1}$ for some $\pi\in G$ then 
$\RP_{\Theta,G,f}(X)$ has a linear factor over $M$. 
Conversely, if $\RP_{\Theta,G,f}(X)$ has a non-repeated linear factor over $M$ 
then there exists $\pi\in G$ such that $\Gal(f)\leq \pi H\pi^{-1}$. 
\end{corollary}

\begin{remark}\label{remGir}
When the resolvent polynomial $\RP_{\Theta,G,f}(X)$ has a repeated factor, there 
always exists a suitable Tschirnhausen transformation $\hat{f}$ of $f$ over $M$ 
(resp. $X-\hat{\Theta}$ of $X-\Theta$ over $k$) such that $\RP_{\Theta,G,\hat{f}}(X)$ 
(resp. $\RP_{\hat{\Theta},G,f}(X)$) has no repeated factors (cf. \cite{Gir83}, 
\cite[Alg. 6.3.4]{Coh93}, \cite{Col95}). 
\end{remark}

In the case where $\RP_{\Theta,G,f}(X)$ has no repeated factors, we have the following theorem: 

\begin{theorem}
For $H\leq G\leq \cS_m$, let $\Theta$ be a $G$-primitive $H$-invariant. 
We assume $\Gal(f)\leq G$ and $\RP_{\Theta,G,f}(X)$ has no repeated factors. 
Then the following two assertions hold\,{\rm :}\\
{\rm (i)} For $\pi\in G$, the fixed group of the field $M\bigl(\omega_{f}(\pi(\Theta))\bigr)$ 
corresponds to $\Gal(f)\cap \pi H\pi^{-1}$. 
Indeed the fixed group of $\Spl_M \RP_{\Theta,G,f}(X)$ corresponds to 
$\Gal(f)\cap \bigcap_{\pi\in G}\pi H\pi^{-1}$\,{\rm ;} \\
{\rm (ii)} let $\varphi : G\rightarrow \cS_{[G:H]}$ denote the permutation representation of 
$G$ on the left cosets of $G/H$ given by the left multiplication. Then we 
have a realization of the Galois group of $\Spl_M \RP_{\Theta,G,f}(X)$ 
as a subgroup of $\cS_{[G:H]}$ by $\varphi(\Gal(f))$. \\
\end{theorem}

\subsection{Formal Tschirnhausen transformation}\label{subseTschirn}
~\\

We recall the geometric interpretation of a Tschirnhausen transformation 
which is given in \cite{HM}. 
Let $f(X)$ and $g(X)$ be monic separable polynomials of degree $n$ in $M[X]$ 
and $\alpha_1,\ldots,\alpha_n$ the fixed ordering roots of $f(X)$ in $\overline{M}$. 
A Tschirnhausen transformation of $f(X)$ over $M$ is a polynomial of the form 
\[
g(X)=\prod_{i=1}^n 
\bigl(X-(c_0+c_1\alpha_i+\cdots+c_{n-1}\alpha_i^{n-1})\bigr),\ c_j \in M.
\]
Two polynomials $f(X)$ and $g(X)$ in $M[X]$ are Tschirnhausen equivalent over $M$ if they are 
Tschirnhausen transformations over $M$ of each other. 
For two irreducible separable polynomials $f(X)$ and $g(X)$ in $M[X]$, 
$f(X)$ and $g(X)$ are Tschirnhausen equivalent over $M$ if and only if 
the quotient fields $M[X]/(f(X))$ and $M[X]/(g(X))$ are isomorphic over $M$. 

In order to obtain an answer to the field intersection problem of $k$-generic polynomials 
via multi-resolvent polynomials, we first treat a general polynomial 
whose roots are $n$ indeterminates $x_1,\ldots,x_n$: 
\begin{align*}
f_\bs(X)\, &=\, \prod_{i=1}^n(X-x_i)\, =\, X^n-s_1X^{n-1}+s_2X^{n-2}+\cdots+(-1)^n s_n\ 
\in k[\bs][X]
\end{align*} 
where $k[x_1,\ldots,x_n]^{\cS_n}=k[\bs]:=k[s_1,\ldots,s_n], \bs=(s_1, \ldots, s_n),$ and $s_i$ is the 
$i$-th elementary symmetric function in $n$ variables $\bx=(x_1,\ldots,x_n)$. 

Let $R_\bx:=k[x_1,\ldots,x_n]$ and $R_\by:=k[y_1,\ldots,y_n]$ be polynomial rings over $k$. 
Put $R_{\bx,\by}:=k[\bx,\by,1/\Delta_\bx,1/\Delta_\by]$, where 
$\Delta_\bx:=\prod_{1\leq i<j\leq m}(x_j-x_i)$ and $\Delta_\by:=\prod_{1\leq i<j\leq m}(y_j-y_i)$. 
We define an involution $\iota$ which exchanges the 
indeterminates $x_i$'s and the $y_i$'s: 
\begin{align}
\iota\ :\ R_{\bx,\by}\longrightarrow R_{\bx,\by},\ 
x_i\longmapsto y_i,\ y_i\longmapsto x_i,\quad (i=1,\ldots,n).\label{defiota}
\end{align}
We take another general polynomial $f_\bt(X):=\iota(f_\bs(X))\in k[\bt][X], \bt=(t_1,\ldots,t_n)$ 
with roots $y_1,\ldots,y_n$ where $t_i=\iota(s_i)$ is the $i$-th elementary 
symmetric function in $\by=(y_1,\ldots,y_n)$. 
We put 
\[
K\ :=\ k(\bs,\bt);
\]
it is regarded as the rational function field over $k$ with $2n$ variables. 
For simplicity, we put 
\[
f_{\bs,\bt}(X):=f_\bs(X)f_\bt(X).
\]
The polynomial $f_{\bs,\bt}(X)$ of degree $2n$ is defined over $K$. We denote 
\begin{align*}
\Gs\, :=\, \Gal(f_\bs/K),\quad \Gt\, :=\, \Gal(f_\bt/K),\quad 
\Gst\, :=\, \Gal(f_{\bs,\bt}/K). 
\end{align*}
Then we have $\Gst=\Gs\times\Gt, \Gs\cong \Gt\cong \cS_n$ and $k(\bx,\by)^{\Gst}=K$. 

We intend to apply the results of the previous subsection for 
$m=2n$, $G=\Gst\leq \cS_{2n}$ and $f=f_{\bs,\bt}$. 

Note that over the field $\Spl_K f_{\bs,\bt}(X)=k(\bx,\by)$, there exist $n!$ 
Tschirnhausen transformations from $f_\bs(X)$ to $f_\bt(X)$ with respect to 
$y_{\pi(1)},\ldots,y_{\pi(n)}$ for $\pi\in \cS_n$. 
We study the field of definition of each Tschirnhausen transformation from 
$f_\bs(X)$ to $f_\bt(X)$. 
Let 
\begin{align*}
D:=
\left(
\begin{array}{ccccc}
1 & x_1 & x_1^2 & \cdots & x_1^{n-1}\\ 
1 & x_2 & x_2^2 & \cdots & x_2^{n-1}\\
\vdots & \vdots & \vdots & \ddots & \vdots\\
1 & x_n & x_n^2 & \cdots & x_n^{n-1}\end{array}\right)
\end{align*}
be the Vandermonde matrix of size $n$. 
The matrix $D\in M_n(k(\bx))$ is invertible because the determinant of $D$ equals 
${\rm det}\, D=\Delta_\bx$. 
The field $k(\bs)(\Delta_\bx)$ is a quadratic extension of $k(\bs)$ which corresponds 
to the fixed field of the alternating group of degree $n$. 
We define the $n$-tuple $(u_0(\mathbf{x},\mathbf{y}),\ldots,
u_{n-1}(\mathbf{x},\mathbf{y}))\in (R_{\bx,\by})^n$ by 
\begin{align}
\left(\begin{array}{c}u_0(\mathbf{x},\mathbf{y})\\ u_1(\mathbf{x},
\mathbf{y})\\ \vdots \\ u_{n-1}(\mathbf{x},\mathbf{y})\end{array}\right)
:=D^{-1}\left(\begin{array}{c}y_1\\ y_2\\ \vdots \\ 
y_n\end{array}\right). \label{defu}
\end{align}
It follows from Cramer's rule that 
\begin{align*}
u_i(\mathbf{x},\mathbf{y})=\Delta_\bx^{-1}\cdot\mathrm{det}
\left(\begin{array}{cccccccc}
1 & x_1 & \cdots & x_1^{i-1} & y_1 & x_1^{i+1} & \cdots & x_1^{n-1}\\ 
1 & x_2 & \cdots & x_2^{i-1} & y_2 & x_2^{i+1} & \cdots & x_2^{n-1}\\
\vdots & \vdots & & \vdots & \vdots & \vdots & & \vdots\\
1 & x_n & \cdots & x_n^{i-1} & y_n & x_n^{i+1} & \cdots & x_n^{n-1}
\end{array}\right).
\end{align*}
In order to simplify the presentation, we write 
\[
u_i:=u_i(\mathbf{x},\mathbf{y}),\quad (i=0,\ldots,n-1). 
\]
The Galois group $\Gst$ acts on the orbit $\{\pi(u_i)\ |\ \pi\in \Gst \}$ 
via regular representation from the left. 
However this action is not faithful. 
We put 
\[
\Hst:=\{(\pi_\bx, \pi_\by)\in \Gst\ |\ \pi_\bx(i)=\pi_\by(i)\ \mathrm{for}\ 
i=1,\ldots,n \}\cong \cS_n. 
\]
If $\pi \in \Hst$ then we have $\pi(u_i)=u_i$ for $i=0,\ldots,n-1$. 
Indeed we see the following lemma: 
\begin{lemma}\label{stabil}
For $i$, $0\leq i\leq n-1$, $u_i$ is a $\Gst$-primitive $\Hst$-invariant. 
\end{lemma}

Let $\Theta:=\Theta(\bx,\by)$ be a $\Gst$-primitive $\Hst$-invariant. 
Let $\opi=\pi\Hst$ be a left coset of $\Hst$ in $\Gst$. 
The group $\Gst$ acts on the set $\{ \pi(\Theta)\ |\ \opi\in \Gst/\Hst\}$ 
transitively from the left through the action on the set $\Gst/\Hst$ of left cosets. 
Each of the sets $\{ \overline{(1,\pi_\by)}\ |\ (1,\pi_\by)\in \Gst\}$ 
and $\{ \overline{(\pi_\bx,1)}\ |\ (\pi_\bx,1)\in \Gst\}$ forms a complete residue 
system of $\Gst/\Hst$, and hence the subgroups $\Gs$ and $\Gt$ of $\Gst$ act on the set 
$\{ \pi(\Theta)\ |\ \opi\in \Gst/\Hst\}$ transitively. 
For $\opi=\overline{(1,\pi_\by)}\in \Gst/\Hst$, we obtain the following equality 
from the definition (\ref{defu}): 
\[
y_{\pi_\by(i)} = \pi_\by(u_0)+\pi_\by(u_1) x_i+\cdots+\pi_\by(u_{n-1})x_i^{n-1}\ 
\mathrm{for}\ i=1,\ldots,n. 
\]
The set $\{(\pi(u_0),\ldots,\pi(u_{n-1}))\ |\ \opi\in \Gst/\Hst\}$ 
gives coefficients of $n!$ different Tschirnhausen transformations from $f_\bs(X)$ 
to $f_\bt(X)$ each of which is defined over $K(\pi(u_0),\ldots,\pi(u_{n-1}))$, respectively. 
We call $K(\pi(u_0),\ldots,\pi(u_{n-1})), (\opi\in \Gst/\Hst)$ a field of formal Tschirnhausen 
coefficients from $f_\bs(X)$ to $f_\bt(X)$. 
We put $v_i:=\iota(u_i)$, for $i=0,\ldots,n-1$. 
Then $v_i$ is also a $\Gst$-primitive $\Hst$-invariant, and $K(\pi(v_0),\ldots,\pi(v_{n-1}))$ 
gives a field of formal Tschirnhausen coefficients from $f_\bt(X)$ to $f_\bs(X)$. 
\begin{proposition}\label{prop1}
Let $\Theta$ be a $\Gst$-primitive $\Hst$-invariant. 
Then we have $k(\bx,\by)^{\pi\Hst \pi^{-1}}$ $=$ $K(\pi(u_0),\ldots,\pi(u_{n-1}))$ 
$=$ $K(\pi(\Theta))$ and $[K(\pi(\Theta)) : K]=n!$ for each $\opi\in \Gst/\Hst$. 
\end{proposition}
Hence, for each of the $n!$ fields $K(\pi(\Theta))$, we have 
$\Spl_{K(\pi(\Theta))} f_\bs(X)=\Spl_{K(\pi(\Theta))} f_\bt(X), 
(\opi\in \Gst/\Hst)$. 
We also obtain the following proposition: 
\begin{proposition}\label{propLL}
Let $\Theta$ be a $\Gst$-primitive $\Hst$-invariant. 
Then we have 
\begin{align*}
&{\rm (i)}\ K(\bx)\cap K(\pi(\Theta))=K(\by)\cap K(\pi(\Theta))=K\quad \textrm{for}\quad 
\opi\in \Gst/\Hst\,{\rm ;}\\
&{\rm (ii)}\ K(\bx,\by)=K(\bx,\pi(\Theta))=K(\by,\pi(\Theta))\quad \textrm{for}\quad 
\opi\in \Gst/\Hst\,{\rm ;}\\
&{\rm (iii)}\ K(\bx,\by)=K(\pi(\Theta)\ |\ \opi\in \Gst/\Hst). 
\end{align*}
\end{proposition}

We consider the formal $\Gst$-relative $\Hst$-invariant resolvent polynomial of degree $n!$ 
by $\Theta$: 
\[
\RP_{\Theta,\Gst}(X)=\prod_{\opi\in \Gst/\Hst}(X-\pi(\Theta))\in k(\bs,\bt)[X]. 
\]
It follows from Proposition \ref{prop1} that $\RP_{\Theta,\Gst}(X)$ is irreducible 
over $k(\bs,\bt)$. 
From Proposition \ref{propLL} we have one of the basic results: 
\begin{theorem}\label{th-gen}
The polynomial $\RP_{\Theta,\Gst}(X)$ is $k$-generic for $\cS_n\times \cS_n$. \\
\end{theorem}

\subsection{Field intersection problem $\mathrm{Int}(f_\bs/M)$}\label{subseInt}
~\\

For $\ba=(a_1,\ldots,a_n), \bb=(b_1,\ldots,b_n)\in M^n$, we fix the order of roots 
$\alpha_1,\ldots,\alpha_n$ (resp. $\beta_1,\ldots,\beta_n$) of $f_\ba(X)$ (resp. $f_\bb(X)$) 
in $\overline{M}$. 
Put $f_{\ba,\bb}(X):=f_\ba(X)f_\bb(X)\in M[X]$. 
We denote 
\begin{align*}
L_{\ba} := M(\alpha_1,\ldots,\alpha_n),\quad 
L_{\bb} := M(\beta_1,\ldots,\beta_n).
\end{align*}
Then we have 
\begin{align*}
L_{\ba} = \Spl_{M} f_\ba(X),\quad L_{\bb} = \Spl_{M} f_\bb(X),\quad 
L_{\ba}\,L_{\bb} = \Spl_{M} f_{\ba,\bb}(X).
\end{align*}
We define a specialization homomorphism 
$\omega_{f_{\ba,\bb}}$ by 
\begin{align*}
\omega_{f_{\ba,\bb}} : R_{\bx,\by} &\longrightarrow 
M(\alpha_1,\ldots,\alpha_n,\beta_1,\ldots,\beta_n)=L_{\ba}\,L_{\bb},\\
\Theta(\bx,\by) &\longmapsto\Theta(\alpha_1,\ldots,\alpha_n,\beta_1,\ldots,\beta_n). 
\end{align*}
We put 
\[
D_\ba:=\omega_{f_{\ba,\bb}}(\Delta_\bx^2),\quad D_\bb:=\omega_{f_{\ba,\bb}}(\Delta_\by^2).
\]
We always assume that both of the polynomials $f_\ba(X)$ and $f_\bb(X)$ are separable over $M$, 
i.e. $D_\ba\cdot D_\bb\neq 0$. 
We also put 
\begin{align*}
G_\ba:=\Gal(f_{\ba}/M),\quad G_\bb:=\Gal(f_{\bb}/M),\quad 
G_{\ba,\bb}:=\Gal(f_{\ba,\bb}/M).
\end{align*}
Then we may naturally regard $G_{\ba,\bb}$ as a subgroup of $\Gst$. 
For $\opi \in \Gst/\Hst$, we put 
\begin{align}
c_{i,\pi}:=\omega_{f_{\ba,\bb}}(\pi(u_i)),\quad 
d_{i,\pi}:=\omega_{f_{\ba,\bb}}\bigl(\pi(\iota(u_i))\bigr),\quad (i=0,\ldots,n-1).\label{defc}
\end{align}
Then it follows from the definition (\ref{defu}) of $u_i$ that 
\begin{align*}
\beta_{\pi_\by(i)}\,&=\, c_{0,\pi} + c_{1,\pi}\,\alpha_{\pi_\bx(i)}
+ \cdots + c_{n-1,\pi}\,\alpha_{\pi_\bx(i)}^{n-1},\\
\alpha_{\pi_\bx(i)}\,&=\, d_{0,\pi} + d_{1,\pi}\,\beta_{\pi_\by(i)}
+ \cdots + d_{n-1,\pi}\,\beta_{\pi_\by(i)}^{n-1}
\end{align*}
for each $i = 1, \ldots, n$. 

For each $\opi\in\Gst/\Hst$, there exists a Tschirnhausen transformation from $f_\ba(X)$ 
to $f_\bb(X)$ over the field $M(c_{0,\pi},\ldots,c_{n-1,\pi})$, and 
the $n$-tuple $(d_{0,\pi},\ldots,d_{n-1,\pi})$ gives the coefficients of a transformation 
of the inverse direction. 
From the assumption $D_\ba\cdot D_\bb\neq 0$, we see the following elementary lemmas 
(see, for example, \cite[Lemma 3.1]{HM}): 
\begin{lemma}\label{lemM}
Let $M'/M$ be a field extension. 
For $\ba,\bb \in M^n$ with $D_\ba\cdot D_\bb\neq 0$, 
if $f_\bb(X)$ is a Tschirnhausen transformation of $f_\ba(X)$ over $M'$, then $f_\ba(X)$ 
is a Tschirnhausen transformation of $f_\bb(X)$ over $M'$. 
Indeed we have $M(c_{0,\pi},\ldots,c_{n-1,\pi})=M(d_{0,\pi},\ldots,d_{n-1,\pi})$ 
for every $\opi\in\Gst/\Hst$. 
\end{lemma}

\begin{lemma}\label{lemMM}
For $\ba,\bb \in M^n$ with $D_\ba\cdot D_\bb\neq 0$, the quotient algebras 
$M[X]/(f_\ba(X))$ and $M[X]/(f_\bb(X))$ are $M$-isomorphic 
if and only if there exists $\pi\in \Gst$ such that $M=M(c_{0,\pi},\ldots,c_{n-1,\pi})$. 
\end{lemma}

In order to obtain an answer to $\mathbf{Int}(f_\bs/M)$ we study the $n!$ fields 
$M(c_{0,\pi},\ldots,c_{n-1,\pi})$ of Tschirnhausen coefficients from $f_\ba(X)$ to 
$f_\bb(X)$ over $M$. 
\begin{proposition}[{\cite[Proposition 3.2]{HM}}]\label{propc}
Under the assumption, $D_\ba\cdot D_\bb\neq 0$, we have 
the following two assertions\,{\rm :} 
\begin{align*}
&{\rm (i)}\ \ \Spl_{M(c_{0,\pi},\ldots,c_{n-1,\pi})} f_\ba(X)
=\Spl_{M(c_{0,\pi},\ldots,c_{n-1,\pi})} f_\bb(X)\, \ \textit{for each}\ \,\opi\in\Gst/\Hst\, 
{\rm ;}\\
&{\rm (ii)}\ L_\ba L_\bb=L_\ba\, M(c_{0,\pi},\ldots,c_{n-1,\pi})=L_\bb\, 
M(c_{0,\pi},\ldots,c_{n-1,\pi})\, \ \textit{for each}\ \, \opi\in\Gst/\Hst.
\end{align*}
\end{proposition}
Let $\Theta$ be a $\Gst$-primitive $\Hst$-invariant. 
Applying the specialization $\omega_{f_{\ba,\bb}}$, 
we have a $\Gst$-relative $\Hst$-invariant resolvent polynomial of $f_{\ba,\bb}$ by $\Theta$: 
\begin{align*}
\RP_{\Theta,\Gst,f_{\ba,\bb}}(X)\, =\, \prod_{\opi\in \Gst/\Hst} 
\bigl(X-\omega_{f_{\ba,\bb}}(\pi(\Theta))\bigr)\in M[X]. 
\end{align*}
A polynomial of this kind is called (absolute) multi-resolvent (cf. \cite{GLV88}, 
\cite{RV99}, \cite{Val}). 
\begin{proposition}[{\cite[Proposition 3.7]{HM}}]\label{prop12}
Let $\Theta$ be a $\Gst$-primitive $\Hst$-invariant. 
For $\ba,\bb \in M^n$ with $D_\ba\cdot D_\bb\neq 0$, suppose that the resolvent 
polynomial $\RP_{\Theta,\Gst,f_{\ba,\bb}}(X)$ has no repeated factors. 
Then the following two assertions hold\,{\rm :}\\
$(\mathrm{i})$\ $M(c_{0,\pi},\ldots,c_{n-1,\pi})=M\bigl(\omega_{f_{\ba,\bb}}(\pi(\Theta))\bigr)$ 
for each $\opi\in\Gst/\Hst$\,{\rm ;}\\
$(\mathrm{ii})$\ ${\rm Spl}_M f_{\ba,\bb}(X)
=M(\omega_{f_{\ba,\bb}}(\pi(\Theta))\ |\ \opi\in\Gst/\Hst)$. 
\end{proposition}

We also get the followings (see, for example, \cite[Proposition 3.12, Corollary 3.13]{HM}):

\begin{proposition}\label{propAn}
For $\ba,\bb \in M^n$ with $D_\ba\cdot D_\bb\neq 0$, 
if $\sqrt{D_\ba\cdot D_\bb}\in M$ then the polynomial $\RP_{\Theta,\Gst,f_{\ba,\bb}}$ 
splits into two factors of degree $n!/2$ over $M$ which are not necessary irreducible.
\end{proposition}

\begin{corollary}\label{corAn}
For $\ba,\bb \in M^n$ with $D_\ba\cdot D_\bb\neq 0$, 
if $G_\ba$, $G_\bb\subset \A_n$ then $\RP_{\Theta,\Gst,f_{\ba,\bb}}$ 
splits into two factors of degree $n!/2$ which are not necessary irreducible. 
\end{corollary}

\begin{definition}
For a separable polynomial $f(X)\in k[X]$ of degree $d$, the decomposition type of $f(X)$ 
over $M$, denoted by {\rm DT}$(f/M)$, is defined as the partition of $d$ induced by the 
degrees of the irreducible factors of $f(X)$ over $M$. 
We define the decomposition type {\rm DT}$(\RP_{\Theta,G,f}/M)$ of 
$\RP_{\Theta,G,f}(X)$ over $M$ by {\rm DT}$(\RP_{\Theta,G,\hat{f}}/M)$ where 
$\hat{f}(X)$ is a Tschirnhausen transformation of $f(X)$ over $M$ which satisfies that 
$\RP_{\Theta,G,\hat{f}}(X)$ has no repeated factors (cf. Remark \ref{remGir}). 
\end{definition}

We write $\mathrm{DT}(f):=\mathrm{DT}(f/M)$ for simplicity. 
From Theorem \ref{thfun}, the decomposition type 
$\mathrm{DT}(\RP_{\Theta,\Gst,f_{\ba,\bb}})$ coincides with the partition of $n!$ 
induced by the lengths of the orbits of $\Gst/\Hst$ under the action of $\Gal(f_{\ba,\bb})$. 
Hence, by Proposition \ref{prop12}, $\mathrm{DT}(\RP_{\Theta,\Gst,f_{\ba,\bb}})$ 
gives the degrees of $n!$ fields of Tschirnhausen coefficients $M(c_{0,\pi},\ldots,c_{n-1,\pi})$ 
from $f_\ba(X)$ to $f_\bb(X)$ over $M$; the degree of $M(c_{0,\pi},\ldots,c_{n-1,\pi})$ 
over $M$ is equal to $|\Gal(f_{\ba,\bb})|/|\Gal(f_{\ba,\bb})\cap\pi\Hst\pi^{-1}|$. 

We conclude that the decomposition type of the resolvent polynomial 
$\RP_{\Theta,\Gst,f_{\ba,\bb}}(X)$ over $M$ gives us information 
about the field intersection problem for $f_\bs(X)$ through the degrees of the fields 
of Tschirnhausen coefficients $M(c_{0,\pi},\ldots,c_{n-1,\pi})$ over $M$ 
which is determined by the degeneration of the Galois group $\Gal(f_{\ba,\bb})$ 
under the specialization $(\bs, \bt) \mapsto (\ba, \bb)$. 

\begin{theorem}[{\cite[Theorem 3.8]{HM}}]\label{throotf}
Let $\Theta$ be a $\Gst$-primitive $\Hst$-invariant. 
For $\ba,\bb \in M^n$ with $D_\ba\cdot D_\bb\neq 0$, 
the following conditions are equivalent\,{\rm :}\\
{\rm (i)} The quotient algebras $M[X]/(f_\ba(X))$ and $M[X]/(f_\bb(X))$ are 
$M$-isomorphic\,{\rm ;}\\
{\rm (ii)} The decomposition type ${\rm  DT}(\RP_{\Theta,\Gst,f_{\ba,\bb}})$ 
over $M$ includes $1$.
\end{theorem}

In the case where $G_\ba$ and $G_\bb$ are isomorphic to a transitive subgroup $G$ of $\cS_n$ 
and every subgroups of $G$ with index $n$ are conjugate in $G$, the condition that 
the quotient algebras $M[X]/(f_\ba(X))$ and $M[X]/(f_\bb(X))$ are $M$-isomorphic 
is equivalent to the condition that $\Spl_M f_\ba(X)$ and $\Spl_M f_\bb(X)$ coincide. 
Hence we obtain an answer to the field isomorphism problem via 
the resolvent polynomial $\RP_{\Theta,\Gst,f_{\ba,\bb}}(X)$. 
\begin{corollary}[An answer to $\mathbf{Isom}(f_\bs^G/M)$]\label{cor1}
For $\ba,\bb \in M^n$ with $D_\ba\cdot D_\bb\neq 0$, we assume that 
both of $f_\ba(X)$ and $f_\bb(X)$ are irreducible over $M$, that $G_\ba$ and $G_\bb$ are 
isomorphic to $G$ and that all subgroups of $G$ with index $n$ are conjugate in $G$. 
Then ${\rm  DT}(\RP_{\Theta,\Gst,f_{\ba,\bb}})$ includes $1$ if and only if 
$\Spl_M f_\ba(X)$ and $\Spl_M f_\bb(X)$ coincide. 
\end{corollary}

For subgroups $H_1$ and $H_2$ of $\cS_n$, we obtain a $k$-generic polynomial for 
$H_1\times H_2$ as a generalization of Theorem \ref{th-gen}. 
\begin{theorem}[{\cite[Theorem 3.10]{HM}}]\label{thgen}
Let $M=k(q_1,\ldots,q_l,r_1,\ldots,r_m)$, $(1\leq l,\, m\leq n-1)$ be the rational function 
field over $k$ with $(l+m)$ variables. 
For $\ba\in {k(q_1,\ldots,q_l)}^n, \bb\in {k(r_1,\ldots,r_m)}^n$, we assume that 
$f_\ba(X)\in M[X]$ and $f_\bb(X)\in M[X]$ be $k$-generic polynomials for $H_1$ and $H_2$, 
respectively. 
If $\RP_{\Theta,\Gst,f_{\ba,\bb}}(X)\in M[X]$ has no repeated factors, then 
$\RP_{\Theta,\Gst,f_{\ba,\bb}}(X)$ is a $k$-generic polynomial 
for $H_1\times H_2$ which is not necessary irreducible. 
\end{theorem}

\section{The cases of $\cS_4$ and of $\A_4$}\label{seS4A4}

Let $M$ be an overfield of $k$ of characteristic $\neq 2$. 
We take a $k$-generic polynomial 
\[
f_{s,t}^{\cS_4}(X)=X^4+sX^2+tX+t\in k(s,t)[X]
\]
for $\cS_4$. 
The discriminant of $f_{s,t}^{\cS_4}(X)$ with respect to $X$ is given by 
\[
D_{s,t}:=t(16s^4 - 128s^2t - 4s^3t + 256t^2 + 144st^2 - 27t^3).
\]
For $\ba=(a,b)\in M$, we always assume that $f_\ba^{\cS_4}(X)$ is separable over $M$, 
i.e. $D_\ba\neq 0$. 

From the definition, for a general quartic polynomial 
\[
g_4(X)=X^4+a_1X^3+a_2X^2+a_3X+a_4\in k[X],\quad (a_1,a_2,a_3,a_4\in M),
\] 
there exist $a,b\in M$ such that 
$\Spl_M f_{a,b}^{\cS_4}(X)=\Spl_M g_4(X)$. 
Indeed we may take such $a,b\in M$ as follows: 
The polynomials $g_4(X)$ and 
\begin{align*}
h_4(X):=g_4(X-a_1/4)=X^4+A_2X^2+A_3X+A_4
\end{align*}
have the same splitting field over $M$, where 
\[
A_2=\frac{-3a_1^2 + 8a_2}{8},\ A_3=\frac{a_1^3 - 4a_1a_2 + 8a_3}{8},\ 
A_4=\frac{-3a_1^4 + 16a_1^2a_2 - 64a_1a_3 + 256a_4}{256}. 
\]
If we put 
\[
a:=\frac{A_2A_3^2}{A_4^2}\quad \mathrm{and}\quad b:=\frac{A_3^4}{A_4^3}
\]
then the polynomials $h_4(X)$ and 
\[
X^4+aX^2+bX+b=\Bigl(\frac{A_3}{A_4}\Bigr)^4\cdot h_4\Bigl(\frac{A_4}{A_3}X\Bigr)
\]
have the same splitting field over $M$. 
Hence we see $\Spl_M f_{a,b}^{\cS_4}(X)=\Spl_M g_4(X)$. 

Let $\cS_4$ act on $k(x_1,x_2,x_3,x_4)$ by $\pi(x_i)=x_{\pi(i)}, (\pi\in \cS_4)$. 
We put 
\[
\sigma:=(1234),\quad \rho_1:=(123),\quad \rho_2:=(234),\quad \omega:=(12)\in \cS_4. 
\]
For the field $k(\bx,\by):=k(x_1,\ldots,x_4,y_1,\ldots,y_4)$, 
we take the interchanging involution
\begin{align*}
\iota\ :\ k(\bx,\by)\,\longrightarrow\,k(\bx,\by),\quad 
x_i\longmapsto y_i,\ y_i\longmapsto x_i,\ (i=1,\ldots,4)
\end{align*}
as in (\ref{defiota}).
Put $(\sigma',\rho_1',\rho_2',\omega'):=
(\iota^{-1}\sigma\iota,\iota^{-1}\rho_1\iota,\iota^{-1}\rho_2\iota,\iota^{-1}\omega\iota)$ 
then $\sigma',\rho_1',\rho_2',\omega'\in\mathrm{Aut}_k(k(\by))$. 
For simplicity we write
\begin{align*}
\cS_4&=\langle\sigma,\omega\rangle,&
\cS_4'&=\langle\sigma',\omega'\rangle,&
\cS_4''&=\langle\sigma\sigma',\omega\omega'\rangle,\\
\A_4&=\langle\rho_1,\rho_2\rangle,&
\A_4'&=\langle\rho_1',\rho_2'\rangle,&
\A_4''&=\langle\rho_1\rho_1',\rho_2\rho_2'\rangle.
\end{align*}
Note that $\cS_4''$ $(\cong\cS_4)$ and $\A_4''$ $(\cong\A_4)$ 
are subgroups of $\cS_4\times\cS_4'$. 

We take an $\cS_4\times \cS_4'$-primitive $\cS_4''$-invariant 
\[
P:=x_1y_1+x_2y_2+x_3y_3+x_4y_4
\] 
and we put $f_{\bs,\bs'}^{\cS_4}(X):=f_\bs^{\cS_4}(X)f_{\bs'}^{\cS_4}(X)$ where 
$(\bs,\bs')=(s,t,s',t')$. 
Then we get an $\cS_4\times \cS_4'$-relative $\cS_4''$-invariant resolvent polynomial of 
$f_{\bs,\bs'}^{\cS_4}(X)$ by $P$ as follows: 
\begin{align}
\R_{\bs,\bs'}(X):=\RP_{P,\cS_4\times \cS_4',f_{\bs,\bs'}^{\cS_4}}
=\Bigl(G_{\bs,\bs'}^1(X)\Bigr)^2-D_{\bs}D_{\bs'}\Bigl(G_{\bs,\bs'}^2(X)\Bigr)^2\in 
k(\bs,\bs')[X]\label{polyR}
\end{align}
where
\begin{align*}
G_{\bs,\bs'}^1(X)={}&{}X^{12}-8s{s'}X^{10}-24t{t'}X^9+(11s^2{s'}^2+4t{s'}^2+4s^2{t'}-80t{t'})X^8\\
&+128st{s'}{t'}X^7+c_6X^6-64tv(3s^2u^2+4tu^2+4s^2v-16tv)X^5+\textstyle{\sum_{i=0}^4 c_i X^i},\\
G_{\bs,\bs'}^2(X)=&-5X^6+12s{s'}X^4+8t{t'}X^3+(-9s^2{s'}^2+20t{s'}^2+20s^2{t'}-16t{t'})X^2\\
&-32st{s'}{t'}X+2s^3{s'}^3-8st{s'}^3+9t^2{s'}^3-8s^3{s'}{t'}+32st{s'}{t'}\\
&-4t^2{s'}{t'}+9s^3{t'}^2-4st{t'}^2-(t^2{t'}^2/2)
\end{align*}
and $c_6,c_4,c_3,\ldots,c_0\in k(\bs,\bs')$ are given by
\begin{align*}
c_6=&-2\bigl{[}8st{s'}^3+13t^2{s'}^3-84t^2{s'}{t'}\bigr{]}-28s^3{s'}^3+576st{s'}{t'}
+57t^2{t'}^2,\\
c_4=&\ 8\,\big{[}3s^2t{s'}^4-14t^2{s'}^4+6st^2{s'}^4+304t^2{s'}^2{t'}-4st^2{s'}^2{t'}
-208st^2{t'}^2\bigl{]}\\
&+17s^4{s'}^4-1216s^2t{s'}^2{t'}-3840t^2{t'}^2-380st^2{s'}{t'}^2,\\
c_3=&-8t{t'}\bigl(-2\bigl{[}40st{s'}^3+9t^2{s'}^3+60t^2{s'}{t'}\bigr{]}-12s^3{s'}^3+832st{s'}{t'}
+37t^2{t'}^2\bigr),\\
c_2=&-2\,\big{[}16s^3t{s'}^5-96st^2{s'}^5+9s^2t^2{s'}^5+108t^3{s'}^5+1280st^2{s'}^3{t'}\\
&+168s^2t^2{s'}^3{t'}-288t^3{s'}^3{t'}-1328s^2t^2{s'}{t'}^2+1472t^3{s'}{t'}^2
-270t^3{s'}^2{t'}^2]\\
&-4s^5{s'}^5+768s^3t{s'}^3{t'}+7168st^2{s'}{t'}^2+141s^2t^2{s'}^2{t'}^2+1616t^3{t'}^3,\\
c_1=&\ 8t{t'}\bigl(-8\bigl{[}3s^2t{s'}^4+18t^2{s'}^4+48t^2{s'}^2{t'}+36st^2{s'}^2{t'}
-16st^2{t'}^2\bigr{]}\\
&-s^4{s'}^4+704s^2t{s'}^2{t'}-256t^2{t'}^2+84st^2{s'}{t'}^2\bigr),\\
c_0=&\ \bigl{[}16s^4t{s'}^6-128s^2t^2{s'}^6-4s^3t^2{s'}^6+256t^3{s'}^6+144st^3{s'}^6-27t^4{s'}^6
\\
&+1280s^2t^2{s'}^4{t'}+176s^3t^2{s'}^4{t'}-2048t^3{s'}^4{t'}-1728st^3{s'}^4{t'}+540t^4{s'}^4{t'}
\\
&-704s^3t^2{s'}^2{t'}^2+4096t^3{s'}^2{t'}^2+4864st^3{s'}^2{t'}^2-720t^4{s'}^2{t'}^2+256t^3{s'}^3{t'}^2\\
&+1008st^3{s'}^3{t'}^2
-270t^4{s'}^3{t'}^2-1024st^3{t'}^3+64t^4{t'}^3-72t^4{s'}{t'}^3\bigr{]}\\
&-256s^4t{s'}^4{t'}-4096s^2t^2{s'}^2{t'}^2
-76s^3t^2{s'}^3{t'}^2-704st^3{s'}{t'}^3-(27t^4{t'}^4/2)
\end{align*}
with simplifying notation $\bigl{[}a\bigr{]}:=a+\iota(a)$. 
It follows from the definition of $\iota$ that $\iota (s,t,s',t')=(s',t',s,t)$. 

Note that the polynomial $\R_{\bs,\bs'}(X)$ splits into two factors 
of degree $12$ over the field $k(\bs,\bs')(\sqrt{D_{\bs}D_{\bs'}})$ as 
\begin{align*}
\R_{\bs,\bs'}(X)=
\Bigl(G_{\bs,\bs'}^1(X)+\sqrt{D_{\bs}D_{\bs'}}\,G_{\bs,\bs'}^2(X)\Bigr)
\Bigl(G_{\bs,\bs'}^1(X)-\sqrt{D_{\bs}D_{\bs'}}\,G_{\bs,\bs'}^2(X)\Bigr),
\end{align*}
and one of the two factors of $\R_{\bs,\bs'}(X)$ above is the 
$\A_4\times \A_4'$-relative $\A_4''$-invariant resolvent polynomial 
$\RP_{P,\A_4\times \A_4',f_{\bs,\bs'}^{\cS_4}}(X)$ of $f_{\bs,\bs'}^{\cS_4}(X)$ by $P$. 

For $\ba=(a,b)$, $\ba'=(a',b')\in M^2$ with $D_\ba\cdot D_{\ba'}\neq 0$, we put 
\[
L_\ba:=\Spl_M f_\ba^{\cS_4}(X),\quad 
G_\ba:=\Gal(f_\ba^{\cS_4}/M),\quad 
G_{\ba,\ba'}:=\Gal(f_{\ba,\ba'}^{\cS_4}/M).
\]

By Theorem \ref{thfun}, we get an answer to $\mathbf{Int}(f_{\bs}^{\cS_4}/M)$ via 
$\R_{\bs,\bs'}(X)$. 
Here we treat only the case where both $f_\ba^{\cS_4}(X)$ and $f_{\ba'}^{\cS_4}(X)$ are 
irreducible over $M$ and $G_\ba=\cS_4$ or $\A_4$. 
We will treat the case where $G_\ba\leq \D_4$ (resp. $f_{\ba}^{\cS_4}(X)$ is reducible) 
in Section \ref{seD4} (Table $3$ and Table $4$ in Theorem \ref{thD4}) 
(resp. Section \ref{seRed} (Table $5$ and Table $6$ in Theorem \ref{thred})). 

An answer to $\mathbf{Int}(f_{\bs}^{\cS_4}/M)$ is given as follows: 

\begin{theorem}\label{thS4A4}
For $\ba=(a,b)$, $\ba'=(a',b')\in M^2$ with $D_\ba\cdot D_{\ba'}\neq 0$, assume that 
both of $f_\ba^{\cS_4}(X)$ and $f_{\ba'}^{\cS_4}(X)$ are irreducible over $M$, 
$\#G_{\ba}\geq \#G_{\ba'}$ and $G_{\ba}=\cS_4$ or $\A_4$. 
If $G_{\ba}=\cS_4$ $($resp. $G_{\ba}=\A_4$$)$ then an answer to the 
intersection problem of $f_{s,t}^{S_4}(X)$ is given by Table $1$ 
according to the decomposition types ${\rm DT}(\R_{\ba,\ba'})$. 
\end{theorem}
\begin{center}
{\rm Table} $1$\vspace*{3mm}\\
{\small 
\begin{tabular}{|c|c|l|l|c|l|l|}\hline
$G_\ba$& $G_{\ba'}$ & & GAP ID & $G_{\ba,{\ba'}}$ & & ${\rm DT}(\R_{\ba,\ba'})$
\\ \hline 
& & (I-1) & $[576,8653]$ & $\cS_4\times \cS_4$ & $L_\ba\cap L_{\ba'}=M$ 
& $24$\\ \cline{3-7} 
& \raisebox{-1.6ex}[0cm][0cm]{$\cS_4$} & (I-2) & $[288,1026]$ 
& $(\A_4\times \A_4)\rtimes \C_2$ 
& $[L_\ba\cap L_{\ba'}:M]=2$ & $12,12$\\ \cline{3-7} 
& & (I-3) & $[96,227]$ & $(\V_4\times \V_4)\rtimes S_3$ 
& $[L_\ba\cap L_{\ba'}:M]=6$ & $12,8,4$\\ \cline{3-7}
& & (I-4) & $[24,12]$ & $\cS_4$ & $L_\ba=L_{\ba'}$ & $8,6,6,3,1$\\ \cline{2-7}
& $\A_4$ & (I-5) & $[288,1024]$ & $\cS_4\times \A_4$ 
& $L_\ba\cap L_{\ba'}=M$ & $24$\\ \cline{2-7}
& & (I-6) & $[192,1472]$ & $\cS_4\times \D_4$ & $L_\ba\cap L_{\ba'}=M$ 
& $24$\\ \cline{3-7} 
\raisebox{-1.6ex}[0cm][0cm]{$\cS_4$} & \raisebox{-1.6ex}[0cm][0cm]{$\D_4$} 
& (I-7) & $[96,187]$ & $(\A_4\times \C_4)\rtimes \C_2$ 
& $[L_\ba\cap L_{\ba'}:M]=2$ & $24$\\ \cline{3-7} 
& & (I-8) & $[96,195]$ & $(\A_4\times \V_4)\rtimes \C_2$
 & $[L_\ba\cap L_{\ba'}:M]=2$ & $24$\\ \cline{3-7} 
& & (I-9) & $[96,195]$ & $(\A_4\times \V_4)\rtimes \C_2$ 
& $[L_\ba\cap L_{\ba'}:M]=2$ & $12,12$\\ \cline{2-7}
& \raisebox{-1.6ex}[0cm][0cm]{$\C_4$} & (I-10) & $[96,186]$ 
& $\cS_4\times \C_4$ & $L_\ba\neq L_{\ba'}$ & $24$\\ \cline{3-7}
& & (I-11) & $[48,30]$ & $\A_4\rtimes \C_4$ & $[L_\ba\cap L_{\ba'}:M]=2$ 
& $12,12$\\ \cline{2-7}
& \raisebox{-1.6ex}[0cm][0cm]{$\V_4$} & (I-12) & $[96,226]$ & 
$\cS_4\times \V_4$ & $L_\ba\neq L_{\ba'}$ & $24$\\ \cline{3-7}
& & (I-13) & $[48,48]$ & $\cS_4\times \C_2$ & $[L_\ba\cap L_{\ba'}:M]=2$ 
& $24$\\ \cline{1-7}
& & (I-14) & $[144,184]$ & $\A_4\times \A_4$ & $L_\ba\cap L_{\ba'}=M$ & $12,12$\\ \cline{3-7} 
& $\A_4$ & (I-15) & $[48,50]$ & $(\V_4\times \V_4)\rtimes 
\C_3$ & $[L_\ba\cap L_{\ba'}:M]=3$ & $12,4,4,4$\\ \cline{3-7} 
\raisebox{-1.6ex}[0cm][0cm]{$\A_4$} 
& & (I-16) & $[12,3]$ & $\A_4$ & $L_\ba=L_{\ba'}$ & $6,6,4,4,3,1$\\ \cline{2-7}
& $\D_4$ & (I-17) & $[96,197]$ & $\A_4\times \D_4$ 
& $L_\ba\cap L_{\ba'}=M$ & $24$\\ \cline{2-7}
& $\C_4$ & (I-18) & $[48,31]$ & $\A_4\times \C_4$ 
& $L_\ba\cap L_{\ba'}=M$ & $24$\\ \cline{2-7}
& $\V_4$ & (I-19) & $[48,49]$ & $\A_4\times \V_4$ 
& $L_\ba\cap L_{\ba'}=M$ & $12,12$\\ \cline{1-7}
\end{tabular}
}\vspace*{5mm}
\end{center}

We checked the decomposition types by using the computer algebra system GAP \cite{GAP} 
(with the command \texttt{DoubleCosetRepsAndSizes}). 
We note that the cases $\{$(I-6),(I-7),(I-8)$\}$ and $\{$(I-12),(I-13)$\}$ may be 
distinguished by comparing the quadratic extensions of $M$ in the splitting fields. 
In the case where $G_\ba=\cS_4$, the unique quadratic extension of $M$ is given by
\begin{align*}
M\bigl(\sqrt{b(16a^4 - 128a^2b - 4a^3b + 256b^2 + 144ab^2 - 27b^3)}\bigr)
\end{align*}
(see Section \ref{seD4} for the case of $G_{\ba'}=\D_4$). 
\begin{example}\label{exS4A4}
We give some numerical examples of Theorem \ref{thS4A4}. \\

(i) Take $M=\mathbb{Q}$ and $\ba=(0,1)$, $\ba'=(2,1)$. 
Then 
\[
f_\ba^{\cS_4}(X)=X^4+X+1\quad \mathrm{and}\quad f_{\ba'}^{\cS_4}(X)=X^4+2X^3+X+1.
\]
We see that $G_\ba=G_{\ba'}=\cS_4$ and $\R_{\ba,\ba'}(X)$ splits over $\mathbb{Q}$ as 
\begin{align*}
\R_{\ba,\ba'}(X)=&\ (X-3)(X+1)^3(X^6-6X^5+12X^4-8X^3-64X^2+128X-64)\\
&\cdot(X^6+6X^5+24X^4+56X^3+32X^2-32X-256)\\
&\cdot(X^8+6X^6-16X^5-89X^4-48X^3+686X^2-1048X+4233). 
\end{align*}
Hence it follows from Theorem \ref{thS4A4} that 
$\Spl_M f_\ba^{\cS_4}(X)=\Spl_M f_{\ba'}^{\cS_4}(X)$.\\

(ii) Take $M=\mathbb{Q}$ and 
\[
\ba=(0,b),\quad \ba'=(a',b')=(2b,b^2)\quad\mathrm{with}\quad b=\frac{2^6}{3^2}.
\] 
Then we see that $G_\ba=G_{\ba'}=\A_4$ and $\R_{\ba,\ba'}(X)$ splits over $\mathbb{Q}$ as 
\begin{align*}
\R_{\ba,\ba'}(X)=&\ \Bigl(X-\frac{2^6}{3}\Bigr)\Bigl(X+\frac{2^6}{3^2}\Bigr)^3
\Bigl(X^4+\frac{2^{13}}{3^3}X^2-\frac{2^{21}}{3^6}X+\frac{2^{24}}{3^6}\Bigr)
\Bigl(X^4-\frac{2^{21}}{3^{6}}X+\frac{2^{26}}{3^7}\Bigr)\\
&\cdot\Bigl(X^6-\frac{2^7}{3}X^5+\frac{2^{14}}{3^3}X^4-\frac{2^{21}}{3^6}X^3
-\frac{2^{24}}{3^6}X^2+\frac{2^{31}}{3^8}X-\frac{2^{36}}{3^{10}}\Bigr)\\
&\cdot\Bigl(X^6+\frac{2^7}{3}X^5+\frac{2^{15}}{3^3}X^4+\frac{2^{21}\cdot7}{3^6}X^3
+\frac{2^{24}\cdot29}{3^7}X^2+\frac{2^{31}\cdot13}{3^9}X+\frac{2^{36}\cdot19}{3^{12}}\Bigr).
\end{align*}
By Theorem \ref{thS4A4}, we get $\Spl_M f_\ba^{\cS_4}(X)=\Spl_M f_{\ba'}^{\cS_4}(X)$.\\
\end{example}

\subsection{Isomorphism problem of $f_\bs^{\cS_4}(X)=X^4+sX^2+tX+t$}\label{seIsoS4}
~\\

Now we consider the problems $\mathbf{Isom}(f_\bs^{\cS_4}/M)$ and 
$\mathbf{Isom^\infty}(f_\bs^{\cS_4}/M)$. 
By Theorem \ref{thS4A4}, we have a criterion to the field isomorphism problem 
$\mathbf{Isom}(f_\bs^{\cS_4}/M)$ for fixed $\ba$, $\ba'\in M^2$. 
However we may not know when $\R_{\ba,\ba'}(X)$ has a linear factor over $M$. 
In particular, we may not answer to $\mathbf{Isom^\infty}(f_\bs^{\cS_4}/M)$, i.e., 
for a fixed $\ba\in M^2$ whether there exist infinitely many $\ba'\in M^2$ such that 
$\Spl_M f_\ba^{\cS_4}(X)=\Spl_M f_{\ba'}^{\cS_4}(X)$ or not. 

In \cite{HM07}, \cite{HM} we gave an answer to $\mathbf{Isom^\infty}(f_s^{\cS_3}/M)$, 
$\mathbf{Isom^\infty}(f_s^{\C_3}/M)$ by using formal Tschirnhausen transformation 
(cf. Section \ref{sePre}). 
We use the same technique to $\mathbf{Isom^\infty}(f_s^{\cS_4}/M)$. 
Here we explain an outline of the proof and we will give the proof in the next subsection. \\

For $\ba=(a,b)$, $\ba'=(a',b')\in M^2$ with $\ba\neq\ba'$ and $D_\ba\cdot D_{\ba'}\neq 0$, 
we take $c_{i,\pi}=\omega_{f_{\ba,\bb}}(\pi(u_i))$, $(i=0,\ldots,3)$, 
and the field of coefficients 
$M(c_{0,\pi},\ldots,c_{3,\pi})$ of Tschirnhausen transformations from $f_\ba^{\cS_4}(X)$ 
to $f_{\ba'}^{\cS_4}(X)$ as in Subsection \ref{subseTschirn}. 
Then we have
\begin{align}
f_{\ba'}^{\cS_4}(X)=\mathrm{Resultant}_Y
(f_{\ba}^{\cS_4}(Y),X-(c_{0,\pi}+c_{1,\pi}Y+c_{2,\pi}Y^2+c_{3,\pi}Y^3)).\label{eqRR}
\end{align}
By Lemma \ref{lemMM}, the splitting field of $f_{\ba}^{\cS_4}(X)$ and of 
$f_{\ba'}^{\cS_4}(X)$ over $M$ coincide if and only if $M=M(c_{0,\pi},\ldots,c_{3,\pi})$ 
for some $\pi\in\cS_4$ unless $\Gal(f_{\ba}^{\cS_4}/M)=\D_4$. 
Thus we take such $\pi\in\cS_4$, and put 
\[
(x,y,z,w):=(c_{0,\pi},\ldots,c_{3,\pi}).
\]
From the assumption $\ba\neq\ba'$, we see $(z,w)\neq (0,0)$. 
Hence we should consider the two cases (i) $w=0$ and $z\neq 0$, (ii) $w\neq 0$. 
In the case of (i) (resp. of (ii)), we put 
\[
p:=\frac{2y}{z},\quad \Bigl(\mathrm{resp.}\quad u:=\frac{4y}{w},\ v:=\frac{2z}{w}\Bigr).
\]
Then by the equality (\ref{eqRR}) we get the following result:
\begin{theorem}\label{thS4}
For $\ba=(a,b)$, $\ba'=(a',b')\in M^2$ with $\ba\neq\ba'$ and $D_\ba\cdot D_{\ba'}\neq 0$, 
two $M$-algebras $M[X]/(f_\ba^{\cS_4}(X))$ and $M[X]/(f_{\ba'}^{\cS_4}(X))$ 
are $M$-isomorphic if and only if either $\mathrm{(i)}$ there exists $p\in M$ such that 
\begin{align}
a'=\frac{P_{\ba,p} Q_{\ba,p}^2}{R_{\ba,p}^2},\quad 
b'=\frac{Q_{\ba,p}^4}{R_{\ba,p}^3}\label{eq1thS4}
\end{align}
where $P_{\bs,p},Q_{\bs,p},R_{\bs,p}\in M[\bs,p]$ with $\bs=(s,t)$ are given by 
\begin{align*}
P_{\bs,p}&=-2(s^2-4t)+6tp+sp^2,\\
Q_{\bs,p}&=-8t^2-8stp-2(s^2-4t)p^2+tp^3,\\
R_{\bs,p}&=s^4-8s^2t+16t^2+8st^2+2t(s^2-4t)p+s(s^2-4t)p^2-stp^3+tp^4,
\end{align*}
or $\mathrm{(ii)}$ there exist $u,v\in M$ such that 
\begin{align}
a'=\frac{U_{\ba,u,v} V_{\ba,u,v}^2}{W_{\ba,u,v}^2},\quad 
b'=\frac{V_{\ba,u,v}^4}{W_{\ba,u,v}^3}\label{eq2thS4}
\end{align}
where  $U_{\bs,u,v},V_{\bs,u,v},W_{\bs,u,v}\in M[\bs,u,v]$ with $\bs=(s,t)$ are given by 
\begin{align*}
U_{\bs,u,v}&=16s^3-48st-6t^2-8s^2u+16tu+su^2-28stv+6tuv-2s^2v^2+8tv^2,\\
V_{\bs,u,v}&=96s^3t-96st^2+8t^3-64s^2tu+16t^2u+14stu^2-tu^3+32s^4v-160s^2tv\\
&+128t^2v-40st^2v-16s^3uv+64stuv+12t^2uv+2s^2u^2v-8tu^2v-32s^2tv^2\\
&+64t^2v^2+8stuv^2+8t^2v^3,\\
W_{\bs,u,v}&=144s^3t^2+256t^3+144st^3-3t^4-128st^2u-120s^2t^2u-32t^3u+16s^2tu^2\\
&+32t^2u^2+33st^2u^2-8stu^3-3t^2u^3+tu^4+96s^4tv-288s^2t^2v+256t^3v+68st^3v\\
&-64s^3tuv+80st^2uv-18t^3uv+14s^2tu^2v-stu^3v+16s^5v^2-112s^3tv^2+192st^2v^2\\
&+2s^2t^2v^2+120t^3v^2-8s^4uv^2+48s^2tuv^2-64t^2uv^2+s^3u^2v^2-4stu^2v^2-4s^3tv^3\\
&+16st^2v^3+24t^3v^3+2s^2tuv^3-8t^2uv^3+s^4v^4-8s^2tv^4+16t^2v^4+8st^2v^4.
\end{align*}
\end{theorem}
\begin{corollary}[An answer to $\mathbf{Isom}(f_\bs^{\cS_4}/M)$]\label{cor12}
Let $\ba$, $\ba'\in M^2$ be as in Theorem \ref{thS4}. 
We also assume that $\Gal(f_\ba^{\cS_4}/M)\neq \D_4$. 
Then two splitting fields of 
$f_\ba^{\cS_4}(X)=X^4+aX^2+bX+b$ and of $f_{\ba'}^{\cS_4}(X)=X^4+a'X^2+b'X+b'$ over $M$ 
coincide if and only if either $\mathrm{(i)}$ there exists $p\in M$ which satisfies 
$(\ref{eq1thS4})$ or $\mathrm{(ii)}$ there exists a pair of $u,v\in M$ which satisfies 
$(\ref{eq2thS4})$
\end{corollary}
By Theorem \ref{thS4} we obtain an answer to $\mathbf{Isom^\infty}(f_\bs^{\cS_4}/M)$ as follows: 
We use the case (i) of Theorem \ref{thS4} (we may also use (ii) instead of (i)). 
We regard $p$ as an independent parameter over $M$ formally 
and take $f_{\ba'}^{\cS_4}(X)\in M(p)[X]$ where 
\begin{align*}
a'=\frac{P_{\ba,p} Q_{\ba,p}^2}{R_{\ba,p}^2},\quad 
b'=\frac{Q_{\ba,p}^4}{R_{\ba,p}^3}
\end{align*}
as in (\ref{eq1thS4}). 
Then we have $\Spl_{M(p)} f_\ba(X)=\Spl_{M(p)} f_{\ba'}(X)$. 
The discriminant of $f_{\ba'}^{\cS_4}(X)$ with respect to $X$ is given by 
\[
\frac{D_\ba Q_{\ba,p}^{12}S_{\ba,p}^2}{R_{\ba,p}^{12}}
\]
where $S_{\ba,p}=-64b^2+16(a^2-4b)p^2+8sp^4+p^6$. 
Thus for $p\in M$ we have $\Spl_M f_\ba(X)=\Spl_M f_{\ba'}(X)$ unless 
$Q_{\ba,p}R_{\ba,p}S_{\ba,p}=0$. 
Since only finitely many $p\in M$ satisfy $Q_{\ba,p}R_{\ba,p}S_{\ba,p}=0$, 
we have the following corollary: 
\begin{corollary}[An answer to $\mathbf{Isom}^\infty(f_\bs^{\cS_4}/M)$]\label{cor2}
Let $M\supset k$ be an infinite field. 
For $\ba\in M^2$ with $D_\ba\neq 0$, there exist infinitely many $\ba'\in M^2$ 
such that $\Spl_M f_\ba^{\cS_4}(X)=\Spl_M f_{\ba'}^{\cS_4}(X)$. 
\end{corollary}
\begin{remark}
By eliminating the variable $v$ (resp. $u$) from the two equalities in (\ref{eq2thS4}) 
of Theorem \ref{thS4}, 
we get the equation $h=0$ where $h\in M(a,b,a',b')[u]$ is a polynomial in $u$ 
(resp. $v$) of degree $24$. 
This polynomial $h$ coincides with the $\cS_4\times \cS_4'$-relative $\cS_4''$-invariant 
resolvent polynomial of $f_{\ba,\ba'}^{\cS_4}(X)$ by $u$ (resp. $v$); for, from the definition 
of $u$ (resp. $v$), we may regard $u=4u_1/u_3$ (resp. $v=2u_2/u_3$) where $u_i$ is the formal 
Tschirnhausen coefficient which is defined in (\ref{defu}). 
Hence from Theorem \ref{thS4} we also get a solution to $\mathbf{Int}(f_{\bs}^{\cS_4}/M)$ by 
using Table $1$ via $\mathrm{DT}(h)$ instead of $\mathrm{DT}(\R_{\ba,\ba'})$. 
\end{remark}

\begin{example}
We give some numerical examples of Theorem \ref{thS4}. 
Note that we always assume $D_\ba\neq 0$ for $\ba=(a,b)\in M^2$. \\

(i) If we take $p=0$ then we have 
\begin{align*}
(P_\ba,Q_\ba,R_\ba)=(-2(a^2-4b),-8b^2,a^4-8a^2b+16b^2+8ab^2). 
\end{align*}
Hence two splitting fields of $f_\ba^{\cS_4}(X)=X^4+aX^2+bX+b$ and of 
\begin{align*}
f_{\ba'}^{\cS_4}(X)=X^4-\frac{2^7(a^2-4b)b^4}{(a^4-8a^2b+16b^2+8ab^2)^2}X^2
+\frac{2^{12}b^8}{(a^4-8a^2b+16b^2+8ab^2)^3}(X+1)
\end{align*}
over $M$ coincide. 
The corresponding Tschirnhausen transformation from $f_{\ba}^{\cS_4}(X)$ to 
$f_{\ba'}^{\cS_4}(X)$ as in (\ref{eqRR}) is given by
\[
f_{\ba'}^{\cS_4}(X)=\mathrm{Resultant}_Y
\Bigl(f_{\ba}^{\cS_4}(Y),X-\Bigl(-\frac{8b^2(a+2Y^2)}{(a^4-8a^2b+16b^2+8ab^2)}\Bigr)\Bigr).
\]
In particular, if we take $a=0$ then we see that the polynomials 
\[
f_{0,b}^{\cS_4}(X)=X^4+bX+b\quad \mathrm{and}\quad 
f_{2b,b^2}^{\cS_4}(X)=X^4+2bX^2+b^2(X+1)
\]
have the same splitting field over $M$. 
We remark that this example is a generalization of Example \ref{exS4A4} (i), (ii). \\

(ii) If we take $p=2$ then we have 
\begin{align*}
\bigl(P_\ba,Q_\ba,R_\ba\bigr)&=\bigl(-2(-2a+a^2-10b),-8(a^2-5b+2ab+b^2),\\
&\hspace*{9mm} 4a^3+a^4+16b-24ab-4a^2b+8ab^2\bigr). 
\end{align*}
Hence two splitting fields of $f_\ba^{\cS_4}(X)=X^4+aX^2+bX+b$ and of 
\begin{align*}
f_{\ba'}^{\cS_4}(X)=X^4&+\frac{2^7(2a-a^2+10b)(a^2-5b+2ab+b^2)^2}
{(4a^3+a^4+16b-24ab-4a^2b+8ab^2)^2}X^2\\
&+\frac{2^{12}(a^2-5b+2ab+b^2)^4}{(4a^3+a^4+16b-24ab-4a^2b+8ab^2)^3}(X+1)
\end{align*}
over $M$ coincide. 
The corresponding Tschirnhausen transformation from $f_{\ba}^{\cS_4}(X)$ to 
$f_{\ba'}^{\cS_4}(X)$ is given by
\[
f_{\ba'}^{\cS_4}(X)=\mathrm{Resultant}_Y
\Bigl(f_{\ba}^{\cS_4}(Y),X-\Bigl(-\frac{8(a^2-5b+2ab+b^2)(a+2Y+2Y^2)}
{4a^3+a^4+16b-24ab-4a^2b+8ab^2}\Bigr)\Bigr).
\]
In particular, if we take $a=0$ then we see that the polynomials 
\[
f_{0,b}^{\cS_4}(X)=X^4+bX+b\ \ \mathrm{and}\ \ 
f_{5b(b-5)^2,b(b-5)^4}^{\cS_4}(X)=X^4+5b(b-5)^2X^2+b(b-5)^4(X+1)
\]
have the same splitting field over $M$. \\

(iii) If we take $u=v=0$ then we have 
\begin{align*}
\bigl(U_\ba,V_\ba,W_\ba\bigr)=\bigl(&2(8a^3-24ab-3b^2),8b(12a^3-12ab+b^2),\\
&\ b^2(144a^3+256b+144ab-3b^2)\bigr).  
\end{align*}
Hence two splitting fields of $f_\ba^{\cS_4}(X)=X^4+aX^2+bX+b$ and of 
\begin{align*}
f_{\ba'}^{\cS_4}(X)=X^4&+\frac{2^7(8a^3-24ab-3b^2)(12a^3-12ab+b^2)^2}
{b^2(144a^3+256b+144ab-3b^2)^2}X^2\\
&+\frac{2^{12}(12a^3-12ab+b^2)^4}{b^2(144a^3+256b+144ab-3b^2)^3}(X+1)
\end{align*}
over $M$ coincide. 
The corresponding Tschirnhausen transformation from $f_{\ba}^{\cS_4}(X)$ to 
$f_{\ba'}^{\cS_4}(X)$ is given by
\[
f_{\ba'}^{\cS_4}(X)=\mathrm{Resultant}_Y
\Bigl(f_{\ba}^{\cS_4}(Y),X-\Bigl(-\frac{8(12a^3-12ab+b^2)(3b+4Y^3)}
{b(144a^3+256b+144ab-3b^2)}\Bigr)\Bigr).
\]
In particular, if we take $a=0$ then we see that the polynomials 
\[
f_{0,b}^{\cS_4}(X)=X^4+bX+b\ \ \mathrm{and}\ \
f_{-6B^2,-8B^3}^{\cS_4}(X)=X^4-6B^2X^2-8B^3(X+1)\ 
\ \mathrm{with}\ \ B=\frac{8b}{3b-256}
\]
have the same splitting field over $M$. 
We give examples in the case of $b,B\in\mathbb{Z}$ in Table $2$. 

\begin{center}
{\rm Table} $2$\vspace*{3mm}\\
\begin{tabular}{|c||c|c|c|c|c|c|}\hline
$b$ & $-256$ & $64$ & $80$ & $84$ & $85$ & $86$\\ \hline
$B$ & $2$ & $-8$ & $-40$ & $-168$ & $-680$ & $344$\\ \hline
$-6B^2$ & $-24$ & $-384$ & $-9600$ & $-169344$ & $-2774400$ & $-710016$\\ \hline
$-8B^3$ & $-64$ & $4096$ & $512000$ & $37933056$ & $2515456000$ & $-325660672$\\ \hline
\end{tabular}
\vspace*{2mm}\\
\begin{tabular}{|c||c|c|c|c|c|}\hline
$b$ & $88$ & $96$ & $128$ & $256$ & $768$ \\ \hline
$B$ & $88$ & $24$ & $8$ & $4$ & $3$\\ \hline
$-6B^2$ & $-46464$ & $-3456$ & $-384$ & $-96$ & $-54$\\ \hline
$-8B^3$ & $-5451776$ & $-110592$ & $-4096$ & $-512$ & $-216$\\ \hline
\end{tabular}\hspace*{2.8cm}
\vspace*{5mm}
\end{center}

We note that $\Gal(f_{0,b}^{\cS_4}/\mathbb{Q})=\cS_4$ for the $b$'s in Table $2$ 
except for $b=-256$, $128$, $768$ and that $\Gal(f_{0,b}^{\cS_4}/\mathbb{Q})=\D_4$ 
for the exceptional cases $b=-256$, $128$, $768$. \\
\end{example}

\subsection{Proof of Theorem \ref{thS4}}\label{seProof}
~\\

By Lemma \ref{lemMM}, two splitting fields of $f_\ba^{\cS_4}(X)$ and of 
$f_{\ba'}^{\cS_4}(X)$ over $M$ coincide if and only if there exist $x,y,z,w\in M$ 
such that 
\begin{align}
f_{\ba'}^{\cS_4}(X)=R'(x,y,z,w,a,b;X)\label{eqR}
\end{align}
where 
\begin{align*}
&R'(x,y,z,w,s,t;X)\\
&:=\mathrm{Resultant}_Y(f_{\ba}^{\cS_4}(Y),X-(x+yY+zY^2+wY^3))\\
&=t^3w^4+3st^2w^3x-t^3w^3x+s^3w^2x^2-3stw^2x^2+3t^2w^2x^2-3twx^3+x^4-2st^2w^3y\\
&+t^3w^3y-s^2tw^2xy-5t^2w^2xy-2s^2wx^2y+4twx^2y+s^2tw^2y^2+2t^2w^2y^2+2stwxy^2\\
&+sx^2y^2-2stwy^3-txy^3+ty^4-t^3w^3z-2s^2tw^2xz+4t^2w^2xz+st^2w^2xz+stwx^2z\\
&-2sx^3z-st^2w^2yz+4stwxyz-3t^2wxyz+3tx^2yz+3t^2wy^2z-4txy^2z+st^2w^2z^2\\
&+t^2wxz^2+s^2x^2z^2+2tx^2z^2-4t^2wyz^2-stxyz^2+sty^2z^2-2stxz^3+t^2xz^3-t^2yz^3\\
&+t^2z^4+\bigl(-3st^2w^3+t^3w^3-2s^3w^2x+6stw^2x-6t^2w^2x+9twx^2-4x^3+s^2tw^2y\\
&+5t^2w^2y+4s^2wxy-8twxy-2stwy^2-2sxy^2+ty^3+2s^2tw^2z-4t^2w^2z-st^2w^2z\\
&-2stwxz+6sx^2z-4stwyz+3t^2wyz-6txyz+4ty^2z-t^2wz^2-2s^2xz^2-4txz^2\\
&+styz^2+2stz^3-t^2z^3\bigr)X+\bigl(s^3w^2-3stw^2+3t^2w^2-9twx+6x^2-2s^2wy+4twy\\
&+sy^2+stwz-6sxz+3tyz+s^2z^2+2tz^2\bigr)X^2+\bigl(3tw-4x+2sz\bigr)X^3+X^4.\\
\end{align*}
We first see that $(z,w)\neq (0,0)$ as follows: 
If we assume $(z,w)=(0,0)$ then we should have
\begin{align*}
R'(x,y,0,0,s,t;X)=x^4&+sx^2y^2-txy^3+ty^4+(-4x^3-2sxy^2+ty^3)X\\
&+(6x^2+sy^2)X^2-4xX^3+X^4.
\end{align*}
By comparing the coefficients of $X^3$ in (\ref{eqR}), 
we obtain $x=0$. 
It also follows that $y=1$ because we see 
$ty^3=ty^4$ by $R'(0,y,0,0,s,t;X)=X^4+sy^2X^2+ty^3X+ty^4$. 
Thus we obtain $\ba=\ba'$ which contradicts the assumption. \\

{\bf (i) The case of $w=0$ and $z\neq 0$.} 
By comparing the coefficients of $X^3$ in (\ref{eqR}), we see $-4x+2sz=0$; 
hence we have $x=sz/2$. 
By direct computation, we then have
\begin{align*}
R'(sz/2,y,z,w,s,t;X)=c_0+c_1X+c_2X^2+X^4
\end{align*}
where
\begin{align*}
c_0&=\bigl(16ty^4-8sty^3z+4s^3y^2z^2-16sty^2z^2+4s^2tyz^3\\
&\quad\ -16t^2yz^3+s^4z^4-8s^2tz^4+16t^2z^4+8st^2z^4\bigr)\big/16,\\
c_1&=ty^3-s^2y^2z+4ty^2z-2styz^2-t^2z^3,\\
c_2&=\bigl(2sy^2+6tyz-s^2z^2+4tz^2\bigr)\big/2.
\end{align*}

Now it follows from (\ref{eqR}) that $c_0=c_1$. 
We put 
\[
p:=\frac{2y}{z}.
\]
Then, by $c_0=c_1$, we get an equation which is linear in $z$. 
From this equation we have 
\begin{align*}
z&=\frac{2(-2p^2s^2+8p^2t+p^3t-8pst-8t^2)}{p^2s^3+s^4+p^4t-4p^2st-p^3st-8s^2t
+2ps^2t+16t^2-8pt^2+8st^2}=:z'.
\end{align*}
Thus we get $\widetilde\bx:=(x,y,z)=(sz'/2,pz'/2,z')$ and 
\begin{align*}
R'(\widetilde\bx,s,t;X)=\frac{Q_{s,t}^4}{R_{s,t}^3}+\frac{Q_{s,t}^4}{R_{s,t}^3}X
+\frac{P_{s,t} Q_{s,t}^2}{R_{s,t}^2}X^2+X^4.
\end{align*}

{\bf (ii) The case of $w\neq 0$.} 
By comparing the coefficients of $X^3$ in (\ref{eqR}), we see $3tw-4x+2sz=0$. 
Hence follows $x=(3tw+2sz)/4$. 
By direct computation, we have
\begin{align*}
R'((3tw+2sz)/4,y,z,w,s,t;X)=C_0+C_1X+C_2X^2+X^4
\end{align*}
where
\begin{align*}
C_0&=\bigl(144s^3t^2w^4+256t^3w^4+144st^3w^4-3t^4w^4-512st^2w^3y-480s^2t^2w^3y\\
&\quad\ -128t^3w^3y+256s^2tw^2y^2+512t^2w^2y^2+528st^2w^2y^2-512stwy^3-192t^2wy^3\\
&\quad\ +256ty^4+192s^4tw^3z-576s^2t^2w^3z+512t^3w^3z+136st^3w^3z-512s^3tw^2yz\\
&\quad\ +640st^2w^2yz-144t^3w^2yz+448s^2twy^2z-128sty^3z+64s^5w^2z^2-448s^3tw^2z^2\\
&\quad\ +768st^2w^2z^2+8s^2t^2w^2z^2+480t^3w^2z^2-128s^4wyz^2+768s^2twyz^2\\
&\quad\ -1024t^2wyz^2+64s^3y^2z^2-256sty^2z^2-32s^3twz^3+128st^2wz^3+192t^3wz^3\\
&\quad\ +64s^2tyz^3-256t^2yz^3+16s^4z^4-128s^2tz^4+256t^2z^4+128st^2z^4\bigr)\big/256,\\
C_1&=\bigl(-12s^3tw^3+12st^2w^3-t^3w^3+32s^2tw^2y-8t^2w^2y-28stwy^2+8ty^3\\
&\quad\ -8s^4w^2z+40s^2tw^2z-32t^2w^2z+10st^2w^2z+16s^3wyz-64stwyz\\
&\quad\ -12t^2wyz -8s^2y^2z+32ty^2z+16s^2twz^2-32t^2wz^2-16styz^2-8t^2z^3\bigr)\big/8,\\
C_2&=\bigl(8s^3w^2-24stw^2-3t^2w^2-16s^2wy+32twy+8sy^2-28stwz+24tyz\\
&\quad\ -4s^2z^2+16tz^2\bigr)\big/8.
\end{align*}

Now it follows from (\ref{eqR}) that $C_0=C_1$. 
We put 
\[
u:=\frac{4y}{w},\quad v:=\frac{2z}{w}.
\]
Then, by $C_0=C_1$, we get an equation which is linear in $w$. 
From this equation we have 
\begin{align*}
w&=-4\bigl(96s^3t-96st^2+8t^3-64s^2tu+16t^2u+14stu^2-tu^3+32s^4v-160s^2tv\\
&+128t^2v-40st^2v-16s^3uv+64stuv+12t^2uv+2s^2u^2v-8tu^2v-32s^2tv^2\\
&+64t^2v^2+8stuv^2+8t^2v^3\bigr)\big/\bigl(144s^3t^2+256t^3+144st^3-3t^4-128st^2u\\
&-120s^2t^2u-32t^3u+16s^2tu^2+32t^2u^2+33st^2u^2-8stu^3-3t^2u^3+tu^4\\
&+96s^4tv-288s^2t^2v+256t^3v+68st^3v-64s^3tuv+80st^2uv-18t^3uv\\
&+14s^2tu^2v-stu^3v+16s^5v^2-112s^3tv^2+192st^2v^2+2s^2t^2v^2+120t^3v^2\\
&-8s^4uv^2+48s^2tuv^2-64t^2uv^2+s^3u^2v^2-4stu^2v^2-4s^3tv^3+16st^2v^3\\
&+24t^3v^3+2s^2tuv^3-8t^2uv^3+s^4v^4-8s^2tv^4+16t^2v^4+8st^2v^4\bigr)=:w'.
\end{align*}
We finally have $\widetilde\bx:=(x,y,z,w)=((3t+2sv)w'/4,uw'/4,vw'/2,w')$ and 
\begin{align*}
R'(\widetilde\bx,s,t;X)=\frac{V_{s,t}^4}{W_{s,t}^3}+\frac{V_{s,t}^4}{W_{s,t}^3}X
+\frac{U_{s,t} V_{s,t}^2}{W_{s,t}^2}X^2+X^4.\quad \qed
\end{align*}

\section{The case of $\D_4$}\label{seD4}

Let $M$ be an infinite overfield of $k$ with char $k\neq 2$. 
We take a $k$-generic polynomial 
\[
f_{s,t}^{\D_4}(X)=X^4+sX^2+t\in k(s,t)[X].
\]
The discriminant of $f_{s,t}^{\D_4}(X)$ with respect to $X$ is given by 
\[
D_{s,t}:=16t(s^2-4t)^2.
\]
We always assume that for $\ba=(a,b)\in M^2$, $f_\ba^{\D_4}(X)$ is separable over $M$, 
i.e. $D_\ba\neq 0$. \\

\subsection{Transformation to $X^4+sX^2+t$}\label{seTran}
~\\

From the definition of generic polynomial, for a separable quartic polynomial 
\[
g_4(X)=X^4+a_1X^3+a_2X^2+a_3X+a_4\in M[X],\quad (a_1,a_2,a_3,a_4\in M),
\] 
with $\Gal(g_4/M)\leq \D_4$, there exist $a,b\in M$ such that 
$\Spl_M f_{a,b}^{\D_4}(X)=\Spl_M g_4(X)$. \\

Indeed, in 1928, Garver \cite{Gar28-1} proved that $g_4(X)$ $(a_1=0)$ can be transformed into 
the form $f_{a,b}^{\D_4}(X)$ by certain Tschirnhausen transformation for $k=\mathbb{Q}$. 
The aim of this subsection is to give an explicit formula of such a transformation for general 
$g_4(X)$ (including the case $a_1\neq 0$) via resolvent polynomial. 

Let $k(\bx):=k(x_1,\ldots,x_4)$ be the rational function field  over $k$ with variables 
$x_1,\ldots,x_4$ as in Section \ref{sePre}. 
We put 
\[
z_1:=x_1-x_3,\qquad z_2:=x_2-x_4;
\]
then the group $\D_4=\langle\sigma,\tau_1\rangle$ 
where $\sigma=(1234)$ and $\tau_1=(13)$ acts on $k(z_1,z_2)\subset k(x_1,\ldots,x_4)$ as 
\[
\sigma\,:\, z_1\mapsto z_2,\quad z_2\mapsto -z_1,\quad 
\tau_1\,:\, z_1\mapsto -z_1,\quad z_2\mapsto z_2.
\]
Take an $\cS_4$-primitive $\D_4$-invariant $\theta:=x_1x_3+x_2x_4$. 
Then we have $k(\bx)^{\D_4}=k(s_1,\ldots,s_4)(\theta)$ where $s_i$ is the $i$-th 
elementary symmetric functions in $\bx$. 
We consider the minimal polynomial of $z_1$ over $k(s_1,\ldots,s_4)(\theta)$: 
\begin{align*}
&(X-z_1)(X+z_1)(X-z_2)(X+z_2)\\
&=X^4+\bigl(-(x_1^2+x_2^2+x_3^2+x_4^2)+2(x_1x_3+x_2x_4)\bigr)X^2+(x_1-x_3)^2(x_2-x_4)^2\\
&=X^4+(-s_1^2+2s_2+2\theta)X^2+(s_2^2-4s_1s_3+16s_4+2s_2\theta-3\theta^2). 
\end{align*}
Then we have 
\begin{lemma}\label{lemh4}
The polynomials $f_\bs(X)=X^4-s_1X^3+s_2X^2-s_3X+s_4$ and 
\[
X^4+(-s_1^2+2s_2+2\theta)X^2+(s_2^2-4s_1s_3+16s_4+2s_2\theta-3\theta^2)
\]
are Tschirnhausen equivalent over $k(s_1,\ldots,s_4)(\theta)$. 
\end{lemma}
\begin{proof}
It can be checked directly that 
\begin{align*}
x_1=\frac{s_1^2s_2-4s_2^2+4s_1s_3-s_1^2\theta+4\theta^2
+(s_1^3-4s_1s_2+8s_3)z_1+(s_1^2-4s_2+4\theta)z_1^2}{2(s_1^3-4s_1s_2+8s_3)}.
\end{align*}
By the successive actions of $\sigma$ on both sides of this equality, we obtain the assertion. 
\end{proof}

We take an absolute (i.e. $\cS_4$-primitive) $\D_4$-invariant resolvent polynomial of 
$g_4(X)$ by $\theta$: 
\begin{align}
\RP_{\theta,\cS_4,g_4}(X)
&=(X-(x_1x_3+x_2x_4))(X-(x_1x_2+x_3x_4))(X-(x_1x_4+x_2x_3))\label{resS4}\\
&=X^3-a_2X^2+(a_1a_3-4a_4)X-a_3^2-a_1^2a_4^2+4a_2a_4.\nonumber
\end{align}
We note that if $g_4(X)$ is separable over $M$ then $\RP_{\theta,\cS_4,g_4}(X)$ 
is also separable over $M$, because their discriminants exactly coincide. 

From the assumption of $\Gal(g_4/M)\leq \D_4$, 
the resolvent polynomial $\RP_{\theta,\D_4,g_4}(X)$ has a root $c\in M$. 
By specializing parameters $(s_1,s_2,s_3,s_4)\mapsto (-a_1,a_2,-a_3,a_4)\in M^4$ in 
Lemma \ref{lemh4}, we get
\begin{lemma}\label{lemTT}
For $(a_1,a_2,a_3,a_4)\in M^4$, we assume that $a_1^3-4a_1a_2+8a_3\neq 0$. 
Then the two polynomials $g_4(X)=X^4+a_1X^3+a_2X^2+a_3X+a_4$ 
with $\Gal(g_4/M)\leq \D_4$ and 
$f_{a,b}^{\D_4}(X)=X^4+aX^2+b$ are Tschirnhausen equivalent over $M$ 
$($in particular, $\Spl_M g_4(X)=\Spl_M f_{a,b}^{\D_4}(X)$$)$ where 
\begin{align*}
a=-a_1^2+2a_2+2c,\quad b=a_2^2-4a_1a_3+16a_4+2a_2c-3c^2
\end{align*}
and $c\in M$ is a root of $\RP_{\theta,\D_4,g_4}(X)$ as in $(\ref{resS4})$. 
\end{lemma}
\begin{remark}
We may assume $a_1^3-4a_1a_2+8a_3\neq 0$ for our purpose 
because we should treat only the case of $a_1=0$ and $a_3\neq 0$. 
\end{remark}
\begin{example}
Take $M=\mathbb{Q}$ and $(a_1,a_2,a_3,a_4)=(1,1,1,1)$. 
Then we obtain the $5$-th cyclotomic polynomial $g_4(X)=X^4+X^3+X^2+X+1$ 
and the corresponding resolvent polynomial $\RP_{\theta,\D_4,g_4}(X)=X^3-X^2-3X+2$ 
which splits as $(X-2)(X^2+X-1)$ over $\mathbb{Q}$. 
Thus we take $c=2$ to have $(a,b)=(5,5)$. 
Hence it follows that $g_4(X)=X^4+X^3+X^2+X+1$ and $X^4+5X+5$ have 
the same splitting field over $\mathbb{Q}$. \\
\end{example}

\subsection{Intersection problem of $f_\bs^{\D_4}(X)=X^4+sX^2+t$}\label{seIntD4}
~\\

We take the rational function field $k(\bx):=k(x_1,\ldots,x_4)$ over $k$ 
with variables $x_1,\ldots,x_4$ as in Section \ref{sePre}. 
In the case of $\D_4$, by a result of the previous subsection, we may specialize 
$x_3:=-x_1$, $x_4:=-x_2$ and consider the field $k(x_1,x_2)=k(\bx)$. 
Put 
\[
\sigma:=(1234),\quad \tau_1:=(13),\quad \tau_2:=(24),\quad 
\tau_3:=(12)(34),\quad \tau_4:=(14)(23). 
\]
Then the group $\D_4=\langle\sigma,\tau_i\rangle$, $(i=1,\ldots,4)$ acts on $k(x_1,x_2)$ 
as in the previous subsection by 
\begin{align*}
\sigma\,&:\, x_1\mapsto x_2,\, x_2\mapsto -x_1,\\
\tau_1\,&:\,x_1\mapsto -x_1,\, x_2\mapsto x_2,\quad 
\tau_2\,:\, x_1\mapsto x_1,\, x_2\mapsto -x_2,\\
\tau_3\,&:\, x_1\mapsto x_2\mapsto x_1,\hspace*{14mm} 
\tau_4\,:\, x_1\mapsto -x_2,\, x_2\mapsto -x_1.
\end{align*}
We first see that $k(x_1,x_2)^{\D_4}=k(s,t)=:k(\bs)$ where 
\[
s:=-x_1^2-x_2^2,\quad t:=x_1^2x_2^2.
\]
The element $x_1$ (resp. $x_2$) is a $\D_4$-primitive $\langle\tau_2\rangle$-invariant 
(resp. $\langle\tau_1\rangle$-invariant). 
Thus two fields $k(x_1,x_2)^{\langle\tau_2\rangle}=k(\bs)(x_1)$ and 
$k(x_1,x_2)^{\langle\tau_1\rangle}=k(\bs)(x_2)$ are non-Galois quartic fields over 
$k(\bs)=k(s,t)$. 

By Kemper-Mattig's theorem \cite{KM00}, we see that the $\D_4$-primitive 
$\langle\tau_1\rangle$-invariant resolvent polynomial 
\begin{align*}
f_{s,t}^{\D_4}(X)&:=\RP_{x_2,\D_4}(X)=(X^2-x_1^2)(X^2-x_2^2)\\
&\ =X^4+sX^2+t\in k(s,t)[X]
\end{align*}
by $x_2$ is a $k$-generic polynomial for $\D_4$. 
In this section, we treat only the case where $f_\ba^{\D_4}(X)$ is irreducible over $M$. 
(See Section \ref{seRed} for reducible cases.) 

The group $\D_4$ has five elements of order two and they form three $\cS_4$-conjugacy classes 
$\{\tau_1,\tau_2\}$, $\{\tau_3,\tau_4\}$, $\{\sigma^2=\tau_1\tau_2=\tau_3\tau_4\}$, 
and the group $\langle\sigma^2\rangle$ is the center of $\D_4$. 

The element $x_1+x_2$ (resp. $x_1-x_2$) is a $\D_4$-primitive $\langle\tau_3\rangle$-invariant 
(resp. $\langle\tau_4\rangle$-invariant). 
Hence the fields $k(\bx)^{\langle\tau_3\rangle}=k(\bs)(x_1+x_2)$ and 
$k(\bx)^{\langle\tau_4\rangle}=k(\bs)(x_1-x_2)$ are also non-Galois quartic fields over $k(\bs)$. 
Thus the $\D_4$-primitive $\langle\tau_3\rangle$-invariant resolvent polynomial 
\begin{align*}
g_{s,t}^{\D_4}(X)&:=\RP_{x_1+x_2,\D_4}(X)
=\bigl(X^2-(x_1+x_2)^2\bigr)\bigl(X^2-(x_1-x_2)^2\bigr)\\
&\ =X^4+2sX^2+(s^2-4t)\in k(s,t)[X]
\end{align*}
by $x_1+x_2$ is also a $k$-generic polynomial for $\D_4$. 
We see $g_{s,t}^{\D_4}(X)=f_{2s,s^2-4t}^{\D_4}(X)$ and that the discriminant of 
$g_{s,t}^{\D_4}(X)$ with respect to $X$ equals $2^{12}t^2(s^2-4t)$. 

We note that $k(\bs)[X]/(f_{\bs}^{\D_4}(X))$ and $k(\bs)[X]/(g_{\bs}^{\D_4}(X))$ are not 
isomorphic over $k(\bs)$ although $\Spl_{k(\bs)} f_{\bs}^{\D_4}(X)
=\Spl_{k(\bs)} g_{\bs}^{\D_4}(X)=k(x_1,x_2)$. 
From above we see 
\begin{lemma}\label{lemRF}
Assume that $\Gal(f_{\ba}^{\D_4}/M)=\D_4$ for $\ba=(a,b)\in M^2$. 
For $\ba'=(a',b')\in M^2$, the following two conditions are equivalent\,$:$\\
{\rm (i)}\ \ $\Spl_M f_{\ba'}^{\D_4}(X)
=\Spl_M f_{\ba}^{\D_4}(X)\,;$\\
{\rm (ii)} $M[X]/(f_{\ba'}^{\D_4}(X))$ is $M$-isomorphic to either 
$M[X]/(f_{\ba}^{\D_4}(X))$ or $M[X]/(f_{2a,a^2-4b}^{\D_4}(X))$.
\end{lemma}
In the case of $\Gal(f_{\ba}^{\D_4}/M)=\C_4$ or $\V_4$, we see that 
$\Spl_M f_{\ba'}^{\D_4}(X)=\Spl_M f_{\ba}^{\D_4}(X)$ if and only if 
$M[X]/(f_{\ba'}^{\D_4}(X))\cong M[X]/(f_{\ba}^{\D_4}(X))$ (cf. Corollary \ref{cor1}).

The Galois biquadratic field $k(\bx)^{\langle\sigma^2\rangle}$ of $k(\bs)$ is given as 
$k(\bx)^{\langle\sigma^2\rangle}=k(\bs)(x_1/x_2)$ which is obtained as the minimal splitting 
field of 
\begin{align*}
\RP_{x_1/x_2,\D_4}(X)
=\Bigl(X^2-\Bigl(\frac{x_1}{x_2}\Bigr)^2\Bigr)\Bigl(X^2-\Bigl(\frac{x_2}{x_1}\Bigr)^2\Bigr)
=X^4-\frac{s^2-2t}{t}X^2+1
\end{align*}
over $k(\bs)$. 
The group $\D_4$ has three subgroups $\langle\tau_1,\tau_2\rangle$, 
$\C_4={\langle\sigma\rangle}$ and $\langle\tau_3,\tau_4\rangle$ of index two. 
The cyclic group $\C_4=\langle\sigma\rangle$ acts on 
$k(\bx)^{\langle\sigma^2\rangle}=k(\bs)(x_1/x_2)$ by $\sigma\,:\, x_1/x_2\mapsto -x_2/x_1$. 
Hence we take 
\begin{align*}
u:=\frac{x_1}{x_2}-\frac{x_2}{x_1}=\frac{x_1^2-x_2^2}{x_1x_2}=\sqrt{(s^2-4t)/t},\quad 
v:=x_1x_2=\sqrt{t}.
\end{align*}
Then three quadratic fields $k(\bx)^{\langle\sigma\rangle}$, 
$k(\bx)^{\langle\tau_1,\tau_2\rangle}$ and $k(\bx)^{\langle\tau_3,\tau_4\rangle}$ of 
$k(\bs)$ are given as 
\begin{align*}
k(\bx)^{\langle\sigma\rangle}&=k(\bs)(u)=k(\bs)(\sqrt{(s^2-4t)/t}),\\
k(\bx)^{\langle\tau_1,\tau_2\rangle}&=k(\bs)(v)=k(\bs)(\sqrt{t}),\\
k(\bx)^{\langle\tau_3,\tau_4\rangle}&=k(\bs)\bigl((x_1+x_2)(x_1-x_2)\bigr)=k(\bs)(\sqrt{s^2-4t}).
\end{align*}
Note that $t=s^2/(u^2+4)$. 
From the above observation, we see the following three elementary lemmas 
(cf. \cite{Buc1910}, \cite{Gar28-2}, \cite{Les38}, \cite{Plo87}, \cite{KW89}, 
\cite[Chapter 2]{JLY02}):
\begin{lemma}\label{lemgenC4V4}
Let $k$ be a field of char $k\neq 2$. 
Then we have\\
{\rm (i)} $f_{s,u}^{\C_4}(X)=X^4+sX^2+s^2/(u^2+4)\in k(s,u)[X]$ 
is $k$-generic for $\C_4\,;$\\
{\rm (ii)} $f_{s,v}^{\V_4}(X)=X^4+sX^2+v^2\in k(s,v)[X]$ is $k$-generic for $\V_4$.
\end{lemma}
\begin{lemma}
For $\ba=(a,b)\in M^2$ with $D_\ba\neq 0$, the polynomial 
$f_\ba^{\D_4}(X)=X^4+aX^2+b$ is reducible over $M$ 
if and only if either $\sqrt{a^2-4b}\in M$, $\sqrt{-a+2\sqrt{b}}\in M$ or 
$\sqrt{-a-2\sqrt{b}}\in M$. 
\end{lemma}
Note that 
\begin{align*}
g_\ba^{\D_4}(X)&=\RP_{x_1+x_2,\D_4,f_\ba^{\D_4}}(X)=X^4+2aX^2+(a^2-4b)\\
&=\textstyle{\bigl(X-\sqrt{-a+2\sqrt{b}}\bigr)\bigl(X+\sqrt{-a+2\sqrt{b}}\bigr)
\bigl(X-\sqrt{-a-2\sqrt{b}}\bigr)\bigl(X+\sqrt{-a-2\sqrt{b}}\bigr)}.
\end{align*}
\begin{lemma}\label{lem2}
For $\ba=(a,b)\in M^2$ with $D_\ba\neq 0$, assume that 
$f_\ba^{\D_4}(X)=X^4+aX^2+b$ is irreducible over $M$. 
Then the following assertions hold\,{\rm :}\\
$(\mathrm{i})$ $\sqrt{b}\in M$ if and only if $\Gal(f_\ba^{\D_4}/M)=\V_4$\,{\rm; }\\
$(\mathrm{ii})$ $\sqrt{(a^2-4b)/b}\in M$ if and only if $\Gal(f_\ba^{\D_4}/M)=\C_4$\,{\rm ;}\\
$(\mathrm{iii})$ $\sqrt{b}\not\in M$ and $\sqrt{(a^2-4b)/b}\not\in M$ if and only if 
$\Gal(f_\ba^{\D_4}/M)=\D_4$.\\
\end{lemma}

In the case of $\Gal(f_\ba^{\D_4}/M)=\D_4$, three quadratic extensions of $M$ are given as 
\begin{align}
M(\sqrt{b}),\qquad M(\sqrt{(a^2-4b)/b}),\qquad M(\sqrt{a^2-4b}).\label{3quad}
\end{align}

For the field $k(\bx,\by):=k(x_1,x_2,y_1,y_2)$, we take the interchanging involution
\begin{align*}
\iota\ :\ k(\bx,\by)\,\longrightarrow\,k(\bx,\by),\quad 
x_1\longmapsto y_1,\ y_1\longmapsto x_1,\ x_2\longmapsto y_2,\ y_2\longmapsto x_2
\end{align*}
as in Section \ref{seS4A4}.
For $\sigma=(1234)$, $\tau_1=(13)\in \cS_4$, we put 
$(\sigma',\tau_1'):=(\iota^{-1}\sigma\iota,\iota^{-1}\tau_1\iota)$; 
then $\sigma',\tau_1'\in\mathrm{Aut}_k(\by)$; 
and we write 
\begin{align*}
\D_4=\langle\sigma,\tau_1\rangle,\qquad
\D_4'=\langle\sigma',\tau_1'\rangle,\qquad
\D_4''=\langle\sigma\sigma',\tau_1\tau_1'\rangle.
\end{align*}
Note that $\D_4''$ $(\cong \D_4)$ is a subgroup of $\D_4\times\D_4'$. 

Take an $\cS_4\times \cS_4'$-primitive $\cS_4''$-invariant 
$P:=x_1y_1+x_2y_2+x_3y_3+x_4y_4$ as in Section \ref{seS4A4}. 
Put $f_{\bs,\bs'}^{\D_4}(X):=f_\bs^{\D_4}(X)f_{\bs'}^{\D_4}(X)$ where 
$(\bs,\bs')=(s,t,s',t')$. 
Then the $\cS_4\times \cS_4'$-relative $\cS_4''$-invariant resolvent polynomial 
$\R_{\bs,\bs'}(X)$ of $f_{\bs,\bs'}^{\D_4}$ by $P$ splits as 
\begin{align*}
\R_{\bs,\bs'}(X):=
\RP_{P,\cS_4\times \cS_4',f_{\bs,\bs'}^{\D_4}}(X)
=\R_{\bs,\bs'}^{1}(X)\cdot\bigl(\R_{\bs,\bs'}^{2}(X)\bigr)^2
\end{align*}
where
\begin{align}
\R_{\bs,\bs'}^{1}(X):\!
&=\RP_{P,\D_4\times \D_4',f_{\bs,\bs'}^{\D_4}}(X)\nonumber\\
&=X^8-8s{s'}X^6+16(s^2{s'}^2+2t{s'}^2+2s^2{t'}-16t{t'})X^4\label{eqR1}\\
&\quad\ -128s{s'}(t{s'}^2+s^2{t'}-8t{t'})X^2+256(t{s'}^2-s^2{t'})^2,\nonumber\\
\R_{\bs,\bs'}^{2}(X):\!
&=X^8-4s{s'}X^6+2(3s^2{s'}^2-4t{s'}^2-4s^2{t'}-16t{t'})X^4\nonumber\\
&\quad\ -4s{s'}(s^2-4t)({s'}^2-4{t'})X^2+(s^2-4t)^2({s'}^2-4{t'})^2.\nonumber
\end{align}

Note that $P$ is regarded as $\D_4\times \D_4'$-primitive $\D_4''$-invariant in (\ref{eqR1}). 
The discriminant of $\R_{\bs,\bs'}^{1}(X)$ with respect to $X$ is given by 
\begin{align*}
2^{80}t^4{t'}^4(s^2-4t)^4({s'}^2-4{t'})^4(s^2{t'}-{s'}^2t)^2(s^2{s'}^2-4{s'}^2t-4s^2{t'})^4.
\end{align*}

For $\ba=(a,b)$, $\ba'=(a',b')\in M^2$ with $D_\ba\cdot D_{\ba'}\neq 0$, we put 
\[
L_\ba:=\Spl_M f_\ba^{\D_4}(X),\quad 
G_\ba:=\Gal(f_\ba^{\D_4}/M),\quad 
G_{\ba,\ba'}:=\Gal(f_{\ba,\ba'}^{\D_4}/M).
\]
Using Theorem \ref{thfun}, we obtain an answer to $\mathbf{Int}(f_{\bs}^{\D_4}/M)$ 
via resolvent polynomial $\R_{\bs,\bs'}^1(X)$ as follows:

\begin{theorem}\label{thD4}
For $\ba=(a,b)$, $\ba'=(a',b')\in M^2$ with $D_\ba\cdot D_{\ba'}\neq 0$, assume that 
both of $f_\ba^{\D_4}(X)$ and $f_{\ba'}^{\D_4}(X)$ are irreducible over $M$ and 
$\#G_{\ba}\geq \#G_{\ba'}$.
If $G_{\ba}=\D_4$ $($resp. $G_{\ba}=\C_4$ or $\V_4)$ then an answer to the 
intersection problem of $f_{s,t}^{\D_4}(X)$ is given by Table $3$ $($resp. by Table $4$$)$
according to ${\rm DT}(\R_{\ba,\ba'})$. 
\end{theorem}

\begin{center}
{\rm Table} $3$\vspace*{3mm}\\
{\small
\begin{tabular}{|c|c|l|l|c|c|l|l|}\hline
$G_\ba$& $G_{\ba'}$ & & GAP ID & $G_{\ba,{\ba'}}$ & $$ 
& ${\rm DT}(\R_{\ba,\ba'}^{1})$ & ${\rm DT}(\R_{\ba,\ba'})$\\ \hline 
& & (II-1) & $[64,226]$ & $\D_4\times \D_4$ & $L_\ba\cap L_{\ba'}=M$ & $8$ & $16,8$\\ \cline{3-8} 
& & (II-2) & $[32,27]$ & $(\V_4\times \V_4)\rtimes \C_2$ 
& $[L_\ba\cap L_{\ba'}:M]=2$ & $8$ & $16,8$\\ \cline{3-8} 
& & (II-3) & $[32,27]$ & $(\V_4\times \V_4)\rtimes \C_2$ 
& $[L_\ba\cap L_{\ba'}:M]=2$ & $4,4$ & $16,4,4$\\ \cline{3-8}
& & (II-4) & $[32,27]$ & $(\V_4\times \V_4)\rtimes \C_2$ 
& $[L_\ba\cap L_{\ba'}:M]=2$ & $4,4$ & $8,8,4,4$\\ \cline{3-8}
& & (II-5) & $[32,28]$ & $(\C_4\times \V_4)\rtimes \C_2$ 
& $[L_\ba\cap L_{\ba'}:M]=2$ & $8$ & $16,8$\\ \cline{3-8}
& & (II-6) & $[32,34]$ & $(\C_4\times \C_4)\rtimes \C_2$ 
& $[L_\ba\cap L_{\ba'}:M]=2$ & $4,4$ & $16,4,4$\\ \cline{3-8}
& $\D_4$ 
& (II-7) & $[16,3]$ & $(\C_4\times \C_2)\rtimes \C_2$ 
& $[L_\ba\cap L_{\ba'}:M]=4$ & $8$ & $16,8$\\ \cline{3-8} 
& & (II-8) & $[16,3]$ & $(\C_4\times \C_2)\rtimes \C_2$ 
& $[L_\ba\cap L_{\ba'}:M]=4$ & $4,4$ & $16,4,4$\\ \cline{3-8}
& & (II-9) & $[16,3]$ & $(\C_4\times \C_2)\rtimes \C_2$ 
& $[L_\ba\cap L_{\ba'}:M]=4$ & $4,4$ & $8,8,4,4$\\ \cline{3-8}
& & (II-10) & $[16,11]$ & $\D_4\times \C_2$ 
& $[L_\ba\cap L_{\ba'}:M]=4$ & $4,4$ & $16,4,4$\\ \cline{3-8}
& & (II-11) & $[16,11]$ & $\D_4\times \C_2$ 
& $[L_\ba\cap L_{\ba'}:M]=4$ & $2,2,2,2$ & $8^2,2^4$\\ \cline{3-8}
$\D_4$ 
& & (II-12) & $[8,3]$ & $\D_4$ & $L_\ba=L_{\ba'}$ & $4,2,2$ & $8,8,4,2,2$\\ \cline{3-8}
& & (II-13) & $[8,3]$ & $\D_4$ & $L_\ba=L_{\ba'}$ & $2,2,2,1,1$ & $8,4^2,2^3,1^2$\\ \cline{2-8}
& & (II-14) & $[32,25]$ & $\D_4\times \C_4$ & $L_\ba\cap L_{\ba'}=M$ & $8$ & $16,8$\\ \cline{3-8}
& \raisebox{-1.6ex}[0cm][0cm]{$\C_4$} 
  & (II-15) & $[16,3]$ & $(\C_4\times \C_2)\rtimes \C_2$ 
& $[L_\ba\cap L_{\ba'}:M]=2$ & $4,4$ & $16,4,4$\\ \cline{3-8}
& & (II-16) & $[16,3]$ & $(\C_4\times \C_2)\rtimes \C_2$ 
& $[L_\ba\cap L_{\ba'}:M]=2$ & $4,4$ & $8,8,4,4$\\ \cline{3-8}
& & (II-17) & $[16,4]$ & $\C_4\rtimes \C_4$ 
& $[L_\ba\cap L_{\ba'}:M]=2$ & $8$ & $16,8$\\ \cline{2-8}
& & (II-18) & $[32,46]$ & $\D_4\times \V_4$ 
& $L_\ba\cap L_{\ba'}=M$ & $8$ & $8,8,8$\\ \cline{3-8} 
& & (II-19) & $[16,11]$ & $\D_4\times \C_2$ 
& $[L_\ba\cap L_{\ba'}:M]=2$ & $8$ & $8,8,8$\\ \cline{3-8}
& \raisebox{-1.6ex}[0cm][0cm]{$\V_4$} 
  & (II-20) & $[16,11]$ & $\D_4\times \C_2$ 
& $[L_\ba\cap L_{\ba'}:M]=2$ & $8$ & $8,8,4,4$\\ \cline{3-8}
& & (II-21) & $[16,11]$ & $\D_4\times \C_2$ 
& $[L_\ba\cap L_{\ba'}:M]=2$ & $4,4$ & $8,8,4,4$\\ \cline{3-8}
& & (II-22) & $[8,3]$ & $\D_4$ 
& $L_\ba\supset L_{\ba'}$ & $8$ & $8,4,4,4,4$\\ \cline{3-8}
& & (II-23) & $[8,3]$ & $\D_4$ 
& $L_\ba\supset L_{\ba'}$ & $4,4$ & $8,4,4,4,4$\\ \cline{1-8}
\end{tabular}
}\vspace*{5mm}
\end{center}
\begin{center}
{\rm Table} $4$\vspace*{3mm}\\
{\small 
\begin{tabular}{|c|c|l|l|c|c|l|l|}\hline
$G_\ba$& $G_{\ba'}$ & & GAP ID & $G_{\ba,{\ba'}}$ & $[L_\ba\cap L_{\ba'}:M]$ 
& ${\rm DT}(\R_{\ba,\ba'}^{1})$ & ${\rm DT}(\R_{\ba,\ba'})$\\ \hline 
& & (III-1) & $[16,2]$ & $\C_4\times \C_4$ 
& $L_\ba\cap L_{\ba'}=M$ & $4,4$ & $16,4,4$\\ \cline{3-8} 
& $\C_4$ 
& (III-2) & $[8,2]$ & $\C_4\times \C_2$ 
& $[L_\ba\cap L_{\ba'}:M]=2$ & $2,2,2,2$ & $8^2,2^4$\\ \cline{3-8} 
\raisebox{-1.6ex}[0cm][0cm]{$\C_4$} 
& & (III-3) & $[4,1]$ & $\C_4$ 
& $L_\ba=L_{\ba'}$ & $2^2,1^4$ & $4^4,2^2,1^4$\\ \cline{2-8}
& & (III-4) & $[16,10]$ & $\C_4\times \V_4$ 
& $L_\ba\cap L_{\ba'}=M$ & $8$ & $8,8,8$\\ \cline{3-8} 
& $\V_4$ 
& (III-5) & $[8,2]$ & $\C_4\times \C_2$ 
& $[L_\ba\cap L_{\ba'}:M]=2$ & $8$ & $8,8,4,4$\\ \cline{3-8} 
& & (III-6) & $[8,2]$ & $\C_4\times \C_2$ 
& $[L_\ba\cap L_{\ba'}:M]=2$ & $4,4$ & $8,8,4,4$\\ \cline{1-8}
& & (III-7) & $[16,10]$ & $\V_4\times \C_4$ 
& $L_\ba\cap L_{\ba'}=M$ & $8$ & $8,8,8$\\ \cline{3-8} 
& $\C_4$ 
  & (III-8) & $[8,2]$ & $\C_2\times \C_4$ 
& $[L_\ba\cap L_{\ba'}:M]=2$ & $8$ & $8,8,4,4$\\ \cline{3-8} 
& & (III-9) & $[8,2]$ & $\C_2\times \C_4$ 
& $[L_\ba\cap L_{\ba'}:M]=2$ & $4,4$ & $8,8,4,4$\\ \cline{2-8}
& & (III-10) & $[16,14]$ & $\V_4\times \V_4$ 
& $L_\ba\cap L_{\ba'}=M$ & $4,4$ & $4^6$\\ \cline{3-8} 
$\V_4$ & & (III-11) & $[8,5]$ & $\V_4\times \C_2$ 
& $[L_\ba\cap L_{\ba'}:M]=2$ & $4,4$ & $4^4,2^4$\\ \cline{3-8} 
& \raisebox{-1.6ex}[0cm][0cm]{$\V_4$} 
  & (III-12) & $[8,5]$ & $\V_4\times \C_2$ 
& $[L_\ba\cap L_{\ba'}:M]=2$ & $4,2,2$ & $4^4,2^4$\\ \cline{3-8}
& & (III-13) & $[8,5]$ & $\V_4\times \C_2$ 
& $[L_\ba\cap L_{\ba'}:M]=2$ & $2,2,2,2$ & $4^4,2^4$\\ \cline{3-8} 
& & (III-14) & $[4,2]$ & $\V_4$ 
& $L_\ba=L_{\ba'}$ & $4,2,2$ & $4^2,2^6,1^4$\\ \cline{3-8} 
& & (III-15) & $[4,2]$ & $\V_4$ 
& $L_\ba=L_{\ba'}$ & $2^2,1^4$ & $4^2,2^6,1^4$\\ \cline{1-8}
\end{tabular}
}\vspace*{5mm}
\end{center}
\begin{remark}\label{rem6}
By comparing six fields $M(\sqrt{b})$, $M(\sqrt{(a^2-4b)/b})$, $M(\sqrt{a^2-4b})$, 
$M(\sqrt{b'})$, $M(\sqrt{(a'^2-4b')/b'})$ and $M(\sqrt{a'^2-4b'})$ each of which is a 
quadratic extension of $M$ or coincides with $M$ as in (\ref{3quad}) (cf. Lemma \ref{lem2}), 
we may distinguish all cases in Table $3$ and Table $4$ except for 
$\{$(II-7),\ldots,(II-13)$\}$ and $\{$(III-2),(III-3)$\}$. 
\end{remark}
\begin{remark}
In the case of $G_\ba=G_{\ba'}=\D_4$, the decomposition type $4$, $2$, $2$ of 
$\R_{\ba,\ba'}^1(X)$ over $M$ means that the splitting fields $L_\ba$ and $L_{\ba'}$ 
coincide and the quotient fields $M[X]/(f_\ba^{\D_4}(X))$ and $M[X]/(f_{\ba'}^{\D_4}(X))$ 
are not $M$-isomorphic (cf. Theorem \ref{throotf} and Lemma \ref{lemRF}). 
\end{remark}

\subsection{Isomorphism problems of $f_\bs^{\D_4}(X)=X^4+sX^2+t$}\label{seIsoD4}
~\\

We treat the problems $\mathbf{Isom}(f_\bs^{\D_4}/M)$ and 
$\mathbf{Isom}^\infty(f_\bs^{\D_4}/M)$ more explicitly because by Theorem \ref{thD4} 
we can not clearly see the existence of $\ba'\in M^2$, $(\ba'\neq\ba)$ which satisfies 
$\mathrm{Spl}_M(f_\ba^{\D_4}(X))=\mathrm{Spl}_M(f_{\ba'}^{\D_4}(X))$, 
i.e., $\R_{\ba,\ba'}(X)$ has a linear factor or DT($\R_{\ba,\ba'}^1(X))$ is $4,2,2$. 

The problem $\mathbf{Isom}(f_\bs^{\D_4}/M)$ was investigated by van der Ploeg \cite{Plo87} 
in the case $M=\mathbb{Q}$ and $\Gal(f_{\ba}^{\D_4}/\mathbb{Q})=\D_4$ (see Lemma 
\ref{lemD4rf} below) to explain Shanks' incredible identities \cite{Sha74}. 
We study $\mathbf{Isom}(f_\bs^{\D_4}/M)$ for general $M\supset k$ and 
$\Gal(f_{\ba}^{\D_4}/M)\leq \D_4$ via formal Tschirnhausen transformation 
which is given in Section \ref{sePre}. 

For $f_\bs^{\D_4}(X)=X^4+sX^2+t$, the problem $\mathbf{Isom}^\infty(f_\bs^{\D_4}/M)$ 
has a trivial solution because for an arbitrary $c\in M$, $f_{a,b}^{\D_4}(X)$ and 
$f_{a',b'}^{\D_4}(X)=f_{ac^2,bc^4}^{\D_4}(X)=f_{a,b}^{\D_4}(X/c)\cdot c^4$ have 
the same splitting field over the infinite field $M$. 
Indeed, we have $\Spl_M f_{a,b}^{\D_4}(X)=\Spl_M f_{a',b'}^{\D_4}(X)$ for $a',b'\in M$ 
with $a^2b'={a'}^2b$ and $b'/b=c^4, c\in M$. 
Thus for $f_\bs^{\D_4}(X)$ we consider the refined question: 
\begin{center}
{\bf ${\mathbf{Isom}^{\infty}}'(f_\bs^{\D_4}/M)$} : for a given $a,b\in M$, 
are there infinitely many\\
\hspace*{4.4cm} $a',b'\in M$ with $a^2b'-{a'}^2b\neq0$ or $b'/b\neq c^4$, $(c\in M)$\\
\hspace*{2.4cm} such that $\Spl_M f_{a,b}^{\D_4}(X)=\Spl_M f_{a',b'}^{\D_4}(X)$\,? 
\end{center}

We take the formal Tschirnhausen coefficients $u_i=u_i(\bx,\by)\in k(\bx,\by)$, $(i=0,1,2,3)$ 
which is defined in $(\ref{defu})$. 
Then the element $u_i$, $(i=0,1,2,3)$, 
becomes an $\cS_4\times\cS_4'$-primitive $\cS_4''$-invariant 
and we may take the corresponding resolvent polynomials 
\[
F_{\bs,\bs'}^i(X):=\RP_{u_i,\cS_4\times\cS_4',f_{\bs,\bs'}^{\D_4}}(X),\quad 
F_{\bs,\bs'}^{i,1}(X):=\RP_{u_i,\D_4\times\D_4',f_{\bs,\bs'}^{\D_4}}(X),\quad (i=0,1,2,3).
\]
In the latter case, we regard the $u_i$'s as $\D_4\times\D_4'$-primitive $\D_4''$-invariants. 
Now we put
\[
(d,d'):=(s^2-4t,{s'}^2-4t').
\]
Then by the definition we may evaluate the resolvent polynomial $F_{\bs,\bs'}^i(X)$, 
$(i=0,1,2,3)$, which splits as 
\begin{align*}
F_{\bs,\bs'}^0(X)&=X^8\Bigl(X^4+\frac{s^2{s'}}{2d}X^2+\frac{s^4d'}{16d^2}\Bigr)^4,&
F_{\bs,\bs'}^1(X)&=F_{\bs,\bs'}^{1,1}(X)\cdot (F_{\bs,\bs'}^{1,2}(X))^2,\\
F_{\bs,\bs'}^2(X)&=X^8\Bigl(X^4+\frac{2{s'}}{d}Z^2+\frac{d'}{d^2}\Bigr)^4,&
F_{\bs,\bs'}^3(X)&=F_{\bs,\bs'}^{3,1}(X)\cdot (F_{\bs,\bs'}^{3,2}(X))^2
\end{align*}
where 
\begin{align}
F_{\bs,\bs'}^{0,1}(X)&=F_{\bs,\bs'}^{2,1}(X)=X^8,\\
F_{\bs,\bs'}^{1,1}(X)&=X^8-\frac{2s{s'}(s^2-3t)}{dt}X^6\nonumber\\
&+\frac{s^6{s'}^2-6s^4{s'}^2t+9s^2{s'}^2t^2+2{s'}^2t^3+2s^6{t'}-12s^4t{t'}+18s^2t^2{t'}
-16t^3{t'}}{d^2t^2}X^4\nonumber\\
&-\frac{2s{s'}(s^2-3t)({s'}^2t^3+s^6{t'}-6s^4t{t'}+9s^2t^2{t'}-8t^3{t'})}{d^3t^3}X^2\nonumber\\
&+\frac{(-{s'}^2t^3+s^6{t'}-6s^4t{t'}+9s^2t^2{t'})^2}{d^4t^4},\nonumber\nonumber\\
F_{\bs,\bs'}^{1,2}(X)&=X^8-\frac{s{s'}(s^2-3t)}{dt}X^6\nonumber\\
&+\frac{3s^6{s'}^2-18s^4{s'}^2t+27s^2{s'}^2t^2-4{s'}^2t^3-4s^6{t'}+24s^4t{t'}-36s^2t^2{t'}-
16t^3{t'}}{8d^2t^2}X^4\nonumber\\
&-\frac{s{s'}(s^2-3t)(s^2-t)^2d'}{16d^2t^3}X^2+\frac{(s^2-t)^4d'^2}{256d^2t^4},\nonumber\\
F_{\bs,\bs'}^{3,1}(X)&=X^8-\frac{2s{s'}}{dt}X^6
+\frac{s^2{s'}^2+2{s'}^2t+2s^2{t'}-16t{t'}}{d^2t^2}X^4
-\frac{2s{s'}({s'}^2t+s^2{t'}-8t{t'})}{d^3t^3}X^2\nonumber\\
&+\frac{(s^2{t'}-{s'}^2t)^2}{d^4t^4},\nonumber\\
F_{\bs,\bs'}^{3,2}(X)&=X^8-\frac{s{s'}}{dt}X^6
+\frac{3s^2{s'}^2-4{s'}^2t-4s^2{t'}-16t{t'}}{8d^2t^2}X^4
-\frac{s{s'}d'}{16d^2t^3}X^2+\frac{{d'}^2}{256d^2t^4}.\nonumber
\end{align}
The discriminants of $F_{\bs,\bs'}^{1,1}(X)$ and of $F_{\bs,\bs'}^{3,1}(X)$ 
with respect to $X$ are given by
\begin{align*}
\mathrm{disc}(F_{\bs,\bs'}^{1,1}(X))&=\frac{2^{24}{t'}^4(s^2-t)^8({s'}^2-4{t'})^4
(s^6{t'}-6s^4t{t'}+9s^2t^2{t'}-{s'}^2t^3)^2}
{(s^2-4t)^{24}t^{16}}\\
&\quad \cdot (s^6{s'}^2-6s^4{s'}^2t+9s^2{s'}^2t^2-4{s'}^2t^3-4s^6{t'}
+24s^4t{t'}-36s^2t^2{t'})^4,\\
\mathrm{disc}(F_{\bs,\bs'}^{3,1}(X))&=\frac{2^{24}{t'}^4({s'}^2-4{t'})^4
(s^2{t'}-{s'}^2t)^2(s^2{s'}^2-4{s'}^2t-4s^2{t'})^4}{(s^2-4t)^{24}t^{24}}.
\end{align*}

For $\ba,\ba'\in M^2$ with $D_\ba\cdot D_{\ba'}\neq 0$, we 
assume that $f_{\ba}^{\D_4}(X)$ and $f_{\ba'}^{\D_4}(X)$ are irreducible over $M$ 
and we write 
\[
G_\ba:=\Gal(f_\ba^{\D_4}/M),\quad G_{\ba'}:=\Gal(f_{\ba'}^{\D_4}/M).
\]
From Lemma \ref{lemMM}, two fields $M[X]/(f_\ba^{\D_4}(X))$ and 
$M[X]/(f_{\ba'}^{\D_4}(X))$ are $M$-isomorphic if and only if there exist $x,y,z,w\in M$ 
such that 
\begin{align}
f_{\ba'}^{\D_4}(X)&=R'(x,y,z,w,a,b;X)\label{eqD4R}\\
&:=\mathrm{Resultant}_X(f_{\ba}^{\D_4}(X),X-(x+yY+zY^2+wY^3))\nonumber
\end{align}
where
\begin{align}
(x,y,z,w)=\omega_{f_{\ba,\ba'}}(\pi(u_0),\pi(u_1),\pi(u_2),\pi(u_3))\quad 
\mathrm{for\ some}\ \ \pi\in\cS_4\times\cS_4'. 
\end{align}
\begin{lemma}\label{lemTTD4}
For $\ba=(a,b)$, $\ba'=(a',b')\in M^2$ with $D_\ba\cdot D_{\ba'}\neq 0$, we assume that 
$\C_4\leq G_\ba, G_{\ba'}$. 
If $M[X]/(f_{\ba}^{\D_4}(X))\cong_M M[X]/(f_{\ba'}^{\D_4}(X))$ then $(x,z)=(0,0)$. 
Namely $f_{\ba'}^{\D_4}(X)$ is obtained from $f_{\ba}^{\D_4}(X)$ by Tschirnhausen transformation 
of the form $yX+wX^3$. 
\end{lemma}
\begin{proof}
By Theorem \ref{throotf}, two fields 
$M[X]/(f_\ba^{\D_4}(X))$ and $M[X]/(f_{\ba'}^{\D_4}(X))$ are 
$M$-isomorphic if and only if $\mathrm{DT}(F_{\bs,\bs'}^i(X))$ includes $1$. 
It follows from the assumption $\C_4\leq G_\ba, G_{\ba'}$ 
and Tables $3$ and $4$ of Theorem \ref{thD4} that $\mathrm{DT}(F_{\bs,\bs'}^i(X))$ includes $1$ 
if and only if $\mathrm{DT}(F_{\bs,\bs'}^{i,1}(X))$ includes $1$. 
Thus we see that $\pi\in\D_4\times\D_4'$ and $(x,z)=(0,0)$ by 
$F_{\bs,\bs'}^{0,1}(X)=F_{\bs,\bs'}^{2,1}(X)=X^8$. 

We see that $(x,z)=(0,0)$ directly as follows: 
The coefficient of $X^3$ of $R'(x,y,z,w,a,b;X)$ equals $-2(2x-az)$. 
Hence by comparing the coefficient of $X^3$ in (\ref{eqD4R}), we see $x=az/2$. 
We also get
\begin{align*}
&R'(az/2,y,z,w,a,b;X)=\\
&X^4+\bigl((2a^3w^2-6abw^2-4a^2wy+8bwy+2ay^2-a^2z^2+4bz^2)/2\bigr)X^2\\
&\hspace*{5.5mm}-(a^2-4b)(a^2w^2-bw^2-2awy+y^2)zX+(16b^3w^4-32ab^2w^3y+16a^2bw^2y^2\\
&\hspace*{5.5mm}+32b^2w^2y^2-32abwy^3+16by^4+4a^5w^2z^2-28a^3bw^2z^2+48ab^2w^2z^2-8a^4wyz^2\\
&\hspace*{5.5mm}+48a^2bwyz^2-64b^2wyz^2+4a^3y^2z^2-16aby^2z^2+a^4z^4-8a^2bz^4+16b^2z^4)/16.
\end{align*}
Hence, by comparing the coefficient of $X$ in (\ref{eqD4R}), we have 
\[
(a^2-4b)(a^2w^2-bw^2-2awy+y^2)z=0.
\]
It follows from the assumption $D_\ba=16b(a^2-4b)^2\neq 0$ that $a^2-4b\neq 0$. 
If $(a^2w^2-bw^2-2awy+y^2)=0$ and $w\neq 0$, then $b=\bigl((aw-y)/w\bigr)^2$ is square in $M$. 
This contradicts $\C_4\leq G_\ba$. 
If $(a^2w^2-bw^2-2awy+y^2)=0$ and $w=0$ then we have $y=0$. 
This contradicts $\C_4\leq G_{\ba'}$ because 
$f_{\ba'}^{\D_4}(X)=R'(az/2,0,z,0,a,b;X)=(X^2+bz^2-(a^2z^2/4))^2$. 

Hence we conclude that $(x,z)=(0,0)$ from the assumption $\C_4\leq G_\ba, G_{\ba'}$. 
\end{proof}
\begin{remark}
From the proof, we see that Lemma \ref{lemTTD4} is not true in general for $G_\ba=\V_4$, 
because the case where $a^2w^2-bw^2-2awy+y^2=0$ and $w\neq 0$ occurs. 
This case corresponds to the case of (III-14) on Table $4$ of Theorem \ref{thD4}.
\end{remark}

Van der Ploeg \cite{Plo87} showed the following result when $M=\mathbb{Q}$ and $G_\ba=\D_4$. 

\begin{lemma}\label{lemD4rf}
For $\ba=(a,b)$, $\ba'=(a',b')\in M^2$ with $D_\ba\cdot D_{\ba'}\neq 0$, we assume that 
$\C_4\leq G_\ba, G_{\ba'}$. 
Then the following conditions are equivalent\,{\rm :}\\
{\rm (i)} $M[X]/(X^4+aX^2+b)\cong_M M[X]/(X^4+a'X^2+b')\,;$\\
{\rm (ii)} there exist $y,w\in M$ such that 
\[
a'= a^3w^2-3abw^2-2a^2wy+4bwy+ay^2,\quad b'=b(bw^2-awy+y^2)^2\,;
\]
{\rm (iii)} there exist $y',w\in M$ such that 
\[
a'=abw^2-4bw{y'}+a{y'}^2,\quad b'=b(bw^2-aw{y'}+{y'}^2)^2.
\]
Moreover, if $a^2b'-{a'}^2b\neq0$ or $b'/b$ is not a fourth power in $M$, 
then the conditions above are equivalent to\\
{\rm (iv)} there exist $u,w\in M$ such that 
\[
a'= (a^3-3ab-2a^2u+4bu+au^2)w^2,\quad b'=b(b-au+u^2)^2w^4.
\] 
\end{lemma}
\begin{proof}
The equivalence of (i) and (ii) follows from Lemma \ref{lemTTD4} and
\begin{align*}
&R'(0,y,0,w,a,b;X)\\
&=X^4+(a^3w^2-3abw^2-2a^2wy+4bwy+ay^2)X^2+b(bw^2-awy+y^2)^2.
\end{align*}
By putting $y':=aw-y$, the equivalence of (ii) and (iii) follows. 
If $a^2b'-{a'}^2b\neq0$ or $b'/b\neq c^4$, $(c\in M)$ then $w\neq 0$. 
Hence the condition (iv) is obtained by putting $u:=y/w$ in (ii). 
\end{proof}
By Lemma \ref{lemRF} and Lemma \ref{lemD4rf} we obtain an answer to 
$\mathbf{Isom}(f_\bs^{\D_4}/M)$:
\begin{proposition}[An answer to $\mathbf{Isom}(f_\bs^{\D_4}/M)$]\label{lemD4spl}
For $\ba=(a,b)$, $\ba'=(a',b')\in M^2$ with $D_\ba\cdot D_{\ba'}\neq 0$, 
we assume that $G_\ba=\D_4$. 
Then $\Spl_M (X^4+aX^2+b)=\Spl_M (X^4+a'X^2+b')
$ if and only if 
there exist $p,q\in M$ such that either 
\begin{align*}
a'&=ap^2-4bpq+abq^2,\quad b'=b(p^2-apq+bq^2)^2\quad \mathrm{or}\\
a'&=2(ap^2-4bpq+abq^2),\quad b'=(a^2-4b)(p^2-bq^2)^2.
\end{align*}
\end{proposition}
\begin{proof}
It follows from Lemma \ref{lemRF} that 
$\Spl_M f_{\ba}^{\D_4}(X)=\Spl_M f_{\ba'}^{\D_4}(X)$ if and only if 
either $M[X]/(f_{\ba'}^{\D_4}(X))\cong_M M[X]/(f_{\ba}^{\D_4}(X))$ or 
$M[X]/(f_{\ba'}^{\D_4}(X))\cong_M M[X]/(f_{2a,a^2-4b}^{\D_4}(X))$. 
The former case is obtained by putting $(p,q):=(y',w)$ in Lemma \ref{lemD4rf} (iii) 
and the latter case is obtained by putting $(a,b):=(2a,a^2-4b)$ and $(p,q):=(y-aw,2w)$ 
in Lemma \ref{lemD4rf} (ii). 
\end{proof}

Finally by using Lemma \ref{lemD4rf}, we get an answer to 
${\mathbf{Isom}^{\infty}}'(f_\bs^{\D_4}/M)$ over Hilbertian field $M$ as follows:
\begin{theorem}[An answer to ${\mathbf{Isom}^{\infty}}'(f_\bs^{\D_4}/M)$]\label{thD4Hil}
Let $M\supset k$ be a Hilbertian field. 
For $\ba=(a,b)\in M^2$ with $D_\ba\neq 0$, there exist infinitely many $\ba'=(a',b')\in M^2$ 
which satisfy the condition that $b'/b$ is not a fourth power in $M$ and 
$\Spl_M f_\ba^{\D_4}(X)=\Spl_M f_{\ba'}^{\D_4}(X)$. 
\end{theorem}
\begin{proof}
By Lemma \ref{lemD4rf} (iv), for an arbitrary $u\in M$, $f_\ba^{\D_4}(X)$ and 
$f_{\ba'}^{\D_4}(X)$ with 
\[
a'=(a^3-3ab-2a^2u+4bu+au^2),\quad b'=b(b-au+u^2)^2
\]
have the same splitting field over $M$. 
By Hilbert's irreducibility theorem, there exist infinitely many $u\in M$ such that 
\[
X^2-(b-au+u^2)=X^2-(u^2-a/2)^2+(a^2-4b)/4
\]
is irreducible over $M$ because $a^2-4b\neq 0$. 
For such infinitely many $u\in M$, $b'/b=(b-au+u^2)^2$ is not a fourth power in $M$. 
\end{proof}

\section{The cases of $\C_4$ and of $\V_4$}\label{seC4V4}

We take $k$-generic polynomials 
\begin{align*}
f_{s,u}^{\C_4}(X)&:=X^4+sX^2+\frac{s^2}{u^2+4}\in k(s,u)[X],\\
f_{s,v}^{\V_4}(X)&:=X^4+sX^2+v^2\in k(s,v)[X]
\end{align*}
for $\C_4$ and for $\V_4$ respectively (cf. Lemma \ref{lemgenC4V4}). 
The discriminant of $f_{s,u}^{\C_4}(X)$ (resp. $f_{s,v}^{\V_4}(X)$) with respect to $X$ 
is given by $16\, s^6u^4/(u^2+4)^3$ (resp. $16\,v^2(s+2v)^2(s-2v)^2$). 
We always assume that for $(a,c)\in M^2$ (resp. for $(a,d)\in M^2$), 
$f_{a,c}^{\C_4}(X)$ (resp. $f_{a,d}^{\V_4}(X)$) is well-defined and separable over $M$, 
i.e. $ac(a^2+4)\neq 0$ (resp. $d(a+2d)(a-2d)\neq 0$). 

As in (\ref{eqR1}) of Section \ref{seD4}, we take $\D_4\times \D_4'$-primitive 
$\D_4''$-invariant $P:=x_1y_1+x_2y_2+x_3y_3+x_4y_4$ and $\D_4\times \D_4'$-relative 
$\D_4''$-invariant resolvent polynomial by $P$: 
\begin{align*}
\R_{\bs,\bs'}^{1}(2X)/2^8\!
&=\RP_{P,\D_4\times \D_4',f_{\bs,\bs'}^{\D_4}}(2X)/2^8\nonumber\\
&=X^8-2s{s'}X^6+(s^2{s'}^2+2t{s'}^2+2s^2{t'}-16t{t'})X^4\nonumber\\
&\quad\ -2s{s'}(t{s'}^2+s^2{t'}-8t{t'})X^2+(t{s'}^2-s^2{t'})^2. 
\end{align*}

By specializing parameters $(\bs,\bs')=(s,t,s',t')\mapsto (s,s^2/(u^2+4),s',{s'}^2/({u'}^2+4))$ 
of $\R_{\bs,\bs'}^{1}(2X)/2^8$, we have the following decomposition: 
\begin{align*}
\omega_{f}(\R_{\bs,\bs'}^{1}(2X)/2^8)
=\Bigl(X^4-ss'X^2+\frac{s^2s'^2(u+u')^2}{(u^2+4)(u'^2+4)}\Bigr)
\Bigl(X^4-ss'X^2+\frac{s^2s'^2(u-u')^2}{(u^2+4)(u'^2+4)}\Bigr)
\end{align*}
where $f=f_{s,u}^{\C_4} f_{s',u'}^{\C_4}$. 
By Theorem \ref{thD4}, we obtain an answer to $\mathbf{Isom}(f_{s,u}^{\C_4}/M)$. 

\begin{theorem}[An answer to $\mathbf{Isom}(f_{s,u}^{\C_4}/M)$]\label{thC4}
For $\ba=(a,c)$, $\ba'=(a',c')\in M^2$ with $aa'cc'(c^2+4)({c'}^2+4)\neq 0$, 
we assume that $c\neq\pm c'$ and $c\neq \pm 4/c'$. 
Then the splitting fields of $f_{a,c}^{\C_4}(X)$ and of $f_{a',c'}^{\C_4}(X)$ 
over $M$ coincide if and only if either $f_{A,C^+}^{\C_4}(X)$ or $f_{A,C^-}^{\C_4}(X)$ has 
a linear factor over $M$ where 
\begin{align*}
f_{A,C^\pm}^{\C_4}(X)=X^4-aa'X^2+\frac{a^2a'^2(c\pm c')^2}{(c^2+4)(c'^2+4)}
\quad\mathrm{with}\quad A=-aa',\ C^\pm=\frac{cc'\mp 4}{c\pm c'}. 
\end{align*}
\end{theorem}
\begin{remark}
The discriminant of $f_{A,C^\pm}^{\C_4}(X)$ with respect to $X$ equals
\[
\frac{16\,a^6a'^6(c\pm c')^2(cc'\mp 4)}{(c^2+4)^3(c'^2+4)^3}.
\]
We may assume that $c\neq \pm c'$ and $c\neq\pm 4/c'$ without loss of generality 
as we have explained it in Remark \ref{remGir}. 
\end{remark}
\begin{example}
For $f_{a,c}^{\C_4}(X)=X^4+aX^2+a^2/(c^2+4)$, we first note that 
\[
\Spl_M f_{a,c}^{\C_4}(X)=\Spl_M f_{a,-c}^{\C_4}(X)\quad\mathrm{and}\quad  
\Spl_M f_{a,c}^{\C_4}(X)=\Spl_M f_{ae^2,c}^{\C_4}(X)\quad \mathrm{for}\quad a,c,e\in M.
\]
By Theorem \ref{thC4}, we have 
\begin{align}
\Spl_M f_{a,c}^{\C_4}(X)=\Spl_M f_{(c^2+4)/a,c}^{\C_4}(X)\quad\mathrm{for}\quad 
f_{(c^2+4)/a,c}^{\C_4}(X)=X^4+\frac{c^2+4}{a}X^2+\frac{c^2+4}{a^2}\label{eqC41}
\end{align}
because we have $f_{A,C^+}^{\C_4}(X)=(X-2)(X+2)(X-c)(X+c)$ for $(a',c')=((c^2+4)/a,c)$. 
Although Theorem \ref{thC4} is not applicable to the case of $c'=4/c$, it follows from 
$\Spl_M(X^4+aX^2+b)=\Spl_M(X^4+2aX^2+a^2-4b)$ that 
\begin{align}
\Spl_M f_{a,c}^{\C_4}(X)=\Spl_M  f_{2a,4/c}^{\C_4}(X)\quad\mathrm{for}\quad 
f_{2a,4/c}^{\C_4}(X)=X^4+2aX^2+\frac{a^2c^2}{c^2+4}\label{eqC42}.
\end{align}
By (\ref{eqC41}) and (\ref{eqC42}), we also see the polynomials 
\[
f_{a,c}^{\C_4}(X)\quad\mathrm{and}\quad 
f_{2(c^2+4)/a,4/c}^{\C_4}(X)=X^4+\frac{2(c^2+4)}{a}X^2+\frac{c^2(c^2+4)}{a^2}
\]
have the same splitting field over $M$. 
\end{example}
\begin{example}
We take $M=\mathbb{Q}$ and the simplest quartic polynomial
\[
h_n(X)=X^4-nX^3-6X^2+nX+1\in \mathbb{Q}[X],\quad (n\in\mathbb{Z})
\]
with discriminant $4(n^2+16)^3$ whose Galois group over $\mathbb{Q}$ is isomorphic to 
$\C_4$ except for $n=0,\pm 3$ 
(cf. for example, \cite{Gra77}, \cite{Gra87}, \cite{Laz91}, \cite{LP95}, \cite{Kim04}, 
\cite{HH05}, \cite{Duq07}, \cite{Lou07}, and the references therein). 
By Lemma \ref{lemTT}, we see that $h_n(X)$ and 
\[
H_n(X):=f_{-(n^2+16),n/2}^{\C_4}(X)=X^4-(n^2+16)X^2+4(n^2+16)
\]
have the same splitting field over $\mathbb{Q}$. 
For $n\in\mathbb{Z}$, we may assume $1\leq n$ because $\Spl_\mathbb{Q} H_n(X)
=\Spl_\mathbb{Q} H_{-n}(X)$ and $H_n(X)$ splits over $\mathbb{Q}$ only for $n=0$, $\pm 3$. 
For $1\leq m<n$, we apply Theorem \ref{thC4} to $H_m(X)$, $H_n(X)$ with 
$(a,c,a',c')=(-(m^2+16),m/2,-(n^2+16),n/2)$, then we see 
\begin{align*}
f_{A,C^+}^{\C_4}(X)&=(X-60)(X+60)(X-80)(X+80)\quad\mathrm{for}\quad (m,n)=(2,22),\\
f_{A,C^-}^{\C_4}(X)&=(X-255)(X+255)(X-340)(X+340)\quad\mathrm{for}\quad (m,n)=(1,103),\\
f_{A,C^+}^{\C_4}(X)&=(X-2080)(X+2080)(X-4992)(X+4992)\quad\mathrm{for}\quad (m,n)=(4,956).
\end{align*}
Hence we get 
\[
\Spl_M h_m(X)=\Spl_M h_n(X)\quad \mathrm{for}\quad (m,n)\in \{(1,103),(2,22),(4,956)\}. 
\]
For just two cases $(m,n)=(1,16)$, $(2,8)$, Theorem \ref{thC4} was not applicable. 
However it works for a suitable Tschirnhausen transformation of $H_n(X)$ as in 
Remark \ref{remGir}. 
Indeed we may use (\ref{eqC42}) in the previous example. 

We checked by Theorem \ref{thC4} that for integers $m,n$ in the range $1\leq m<n\leq 10^5$, 
$f_{A,C^\pm}^{\C_4}(X)$ has a linear factor over $\mathbb{Q}$, i.e. 
$\Spl_M h_m(X)=\Spl_M h_n(X)$, only for the values of $(m,n)=(1,103),(2,22),(4,956)$.
\end{example}
\begin{remark}
In the case where the field $M$ includes a primitive $4$th root $i:=e^{2\pi \sqrt{-1}/4}$ 
of unity, the polynomial $g_t^{\C_4}(X):=X^4-t\in k(t)[X]$ is $k$-generic 
for $\C_4$ by Kummer theory. 
Indeed we see that the polynomials $f_{a,c}^{\C_4}(X)=X^4+aX^2+a^2/(c^2+4)$ and 
\[
g_{a^2(c-2i)/(c+2i)}^{\C_4}(X)=X^4-\frac{a^2(c-2i)}{c+2i}
\]
are Tschirnhausen equivalent over $M$ because 
\[
\mathrm{Resultant}_X \Bigl(f_{a,c}^{\C_4}(X),
Y-\Bigl(\frac{(c+i)(c-2i)}{c}X+\frac{c^2+4}{ac}X^3\Bigr)\Bigr)=g_{a^2(c-2i)/(c+2i)}^{\C_4}(Y)
\]
(we may assume that $ac\neq 0$ since $f_{0,c}^{\C_4}(X)=X^4$ and $f_{a,0}^{\C_4}=(X^2+a/2)^2$). 
In this case, for $b,b'\in M$ with $b\cdot b'\neq 0$, 
the splitting fields of $g_b^{\C_4}(X)$ and of $g_{b'}^{\C_4}(X)$ over $M$ coincide 
if and only if the polynomial $(X^4-bb')(X^4-b^3b')$ has a linear factor over $M$. 
\end{remark}

Finally let us check the field isomorphism problem ${\bf Isom}(f_{s,v}^{\V_4}/M)$. 
By specializing parameters $(\bs,\bs')=(s,t,s',t')\mapsto (s,v^2,s',v'^2)$ of 
$\R_{\bs,\bs'}^{1}(X)$, we have
\begin{align*}
\R_{s,v,s',v'}^{\V_4}(X)&:=\omega_{f}(\R_{\bs,\bs'}^{1}(2X)/2^8)\\
&=\bigl(X^4-(ss'+4vv')X^2+(sv'+s'v)^2\bigr)\\
&\ \cdot\bigl(X^4-(ss'-4vv')X^2+(sv'-s'v)^2\bigr)
\end{align*}
where $f=f_{s,v}^{\V_4} f_{s',v'}^{\V_4}$. 
As in the case of $f_{s,u}^{\C_4}(X)$, if $\R_{a,d,a',d'}^{\V_4}(X)$ has a (simple) linear 
factor over $M$ for $a,d,a',d'\in M$ then $\Spl_M f_{a,d}^{\V_4}(X)=\Spl_M f_{a',d'}^{\V_4}(X)$. 
However the converse dose not hold by group theoretical reason (see Table $4$ of 
Theorem \ref{thD4}). 

Although we could not get an answer to ${\bf Isom}(f_{s,v}^{\V_4}/M)$ 
via $\R_{s,v,s',v'}^{\V_4}(X)$, the answer can be 
obtained easily by comparing quadratic subfields (see Remark \ref{rem6}). 

\section{Reducible cases}\label{seRed}

In this section, we treat reducible cases. 
We take the $k$-generic polynomial 
\[
f_{s,t}^{\cS_4}(X)=X^4+sX^2+tX+t\in k(s,t)[X]
\]
for $\cS_4$ with discriminant 
\[
D_{s,t}:=t(16s^4 - 128s^2t - 4s^3t + 256t^2 + 144st^2 - 27t^3)
\]
and the $\cS_4\times \cS_4'$-relative $\cS_4''$-invariant resolvent polynomial $\R_{\bs,\bs'}(X)$ 
of the product $f_{\bs,\bs'}^{\cS_4}(X):=f_{s,t}^{\cS_4}(X)$ $\cdot$ $f_{s',t'}^{\cS_4}(X)$ by 
$P:=x_1y_1+x_2y_2+x_3y_3+x_4y_4$ as in (\ref{polyR}) of Section \ref{seS4A4}. 

For $\ba=(a,b)$, $\ba'=(a',b')\in M^2$ with $D_\ba\cdot D_{\ba'}\neq 0$, we put 
\[
L_\ba:=\Spl_M f_\ba^{\cS_4}(X),\quad G_\ba:=\Gal(f_\ba^{\cS_4}/M),
\quad G_{\ba,\ba'}:=\Gal(f_{\ba,\ba'}^{\cS_4}/M).
\]

We assume that $G_\ba=\cS_3$ or $\C_3$ 
and omit the cases where $\#G_\ba\leq 2$ or $\#G_{\ba'}\leq 2$.

\begin{theorem}\label{thred}
For $\ba=(a,b)$, $\ba'=(a',b')\in M^2$ with $D_\ba\cdot D_{\ba'}\neq 0$, assume that 
$G_{\ba}=\cS_3$ or $\C_3$, and $\#G_{\ba'}\geq 3$. 
If $G_{\ba}=\cS_3$ $($resp. $G_{\ba}=\C_3$$)$ then an answer to the 
intersection problem of $f_{s,t}^{\cS_4}(X)$ is given 
by ${\rm DT}(\R_{\ba,\ba'})$ as Table $5$ $($resp. Table $6$$)$ shows. 
\end{theorem}
\begin{center}
{\rm Table} $5$\vspace*{3mm}\\
{\small 
\begin{tabular}{|c|c|l|l|c|l|l|}\hline
$G_\ba$& $G_{\ba'}$ & & GAP ID & $G_{\ba,{\ba'}}$ & & ${\rm DT}(\R_{\ba,\ba'})$
\\ \hline 
& & (IV-1) & $[144,183]$ & $\cS_3\times \cS_4$ & $L_\ba\cap L_{\ba'}=M$ & $24$\\ \cline{3-7} 
& $\cS_4$ & (IV-2) & $[72,43]$ & $(\C_3\times \A_4)\rtimes \C_2$ 
& $[L_\ba\cap L_{\ba'}:M]=2$ & $12,12$\\ \cline{3-7} 
& & (IV-3) & $[24,12]$ & $\cS_4$ & $L_\ba\subset L_{\ba'}$ & $12,8,4$\\ \cline{2-7}
& $\A_4$ & (IV-4) & $[72,44]$ & $\cS_3\times \A_4$ & $L_\ba\cap L_{\ba'}=M$ & $24$\\ \cline{2-7}
& & (IV-5) & $[48,38]$ & $\cS_3\times \D_4$ & $L_\ba\cap L_{\ba'}=M$ & $24$\\ \cline{3-7} 
& \raisebox{-1.6ex}[0cm][0cm]{$\D_4$} 
& (IV-6) & $[24,6]$ & $\D_{12}$ & $[L_\ba\cap L_{\ba'}:M]=2$ & $12,12$\\ \cline{3-7} 
\raisebox{-1.6ex}[0cm][0cm]{$\cS_3$}
& & (IV-7) & $[24,8]$ & $(\C_3\times \V_4)\rtimes \C_2$ & $[L_\ba\cap L_{\ba'}:M]=2$ 
& $24$\\ \cline{3-7} 
& & (IV-8) & $[24,8]$ & $(\C_3\times \V_4)\rtimes \C_2$ & $[L_\ba\cap L_{\ba'}:M]=2$ 
& $12,12$\\ \cline{2-7}
& \raisebox{-1.6ex}[0cm][0cm]{$\C_4$} & (IV-9) & $[24,5]$ 
& $\cS_3\times \C_4$ & $L_\ba\cap L_{\ba'}=M$ & $24$\\ \cline{3-7}
& & (IV-10) & $[12,1]$ & $\C_3\rtimes \C_4$ & $[L_\ba\cap L_{\ba'}:M]=2$ & $12,12$\\ \cline{2-7}
& & (IV-11) & $[24,14]$ & $\cS_3\times \V_4$ & $L_\ba\cap L_{\ba'}=M$ & $24$\\ \cline{3-7} 
& & (IV-12) & $[24,14]$ & $\cS_3\times \V_4$ & $L_\ba\cap L_{\ba'}=M$ & $12,12$\\ \cline{3-7} 
& $\V_4$ & (IV-13) & $[12,4]$ & $\D_6$ & $[L_\ba\cap L_{\ba'}:M]=2$ & $12,12$\\ \cline{3-7} 
& & (IV-14) & $[12,4]$ & $\D_6$ & $[L_\ba\cap L_{\ba'}:M]=2$ & $12,6,6$\\ \cline{3-7} 
& & (IV-15) & $[12,4]$ & $\D_6$ & $[L_\ba\cap L_{\ba'}:M]=2$ & $6,6,6,6$\\ \hline\hline
& & (IV-16) & $[36,10]$ & $\cS_3\times \cS_3$ & $L_\ba\cap L_{\ba'}=M$ & $18,6$\\ \cline{3-7}  
\raisebox{-1.6ex}[0cm][0cm]{$\cS_3$} & $\cS_3$ & (IV-17) & $[18,4]$ & $(\C_3\times \C_3)\rtimes 
\C_2$ & $[L_\ba\cap L_{\ba'}:M]=2$ 
& $9,9,3,3$\\ \cline{3-7} 
& & (IV-18) & $[6,1]$ & $\cS_3$ & $L_\ba=L_{\ba'}$ & $6^2,3^3,2,1$\\ \cline{2-7}
& $\C_3$ & (IV-19) & $[18,3]$ & $\cS_3\times \C_3$ & $L_\ba\cap L_{\ba'}=M$ & $18,6$\\ \cline{1-7} 
\end{tabular}
}\vspace*{5mm}
\end{center}
\begin{center}
{\rm Table} $6$\vspace*{3mm}\\
{\small 
\begin{tabular}{|c|c|l|l|c|l|l|}\hline
$G_\ba$& $G_{\ba'}$ & & GAP ID & $G_{\ba,{\ba'}}$ & & ${\rm DT}(\R_{\ba,\ba'})$
\\ \hline 
& $\cS_4$ & (V-1) & $[72,42]$ & $\C_3\times \cS_4$ & $L_\ba\cap L_{\ba'}=M$ 
& $24$\\ \cline{2-7}
& \raisebox{-1.6ex}[0cm][0cm]{$\A_4$} & (V-2) & $[36,11]$ & $\C_3\times \A_4$ 
& $L_\ba\cap L_{\ba'}=M$ & $12,12$\\ \cline{3-7} 
\raisebox{-1.6ex}[0cm][0cm]{$\C_3$} 
& & (V-3) & $[12,3]$ & $\A_4$ & $L_\ba\subset L_{\ba'}$ & $12,4,4,4$\\ \cline{2-7}
& $\D_4$ & (V-4) & $[24,10]$ & $\C_3\times \D_4$ & $L_\ba\cap L_{\ba'}=M$ 
& $24$\\ \cline{2-7}
& $\C_4$ & (V-5) & $[12,2]$ & $\C_{12}$ & $L_\ba\cap L_{\ba'}=M$ 
& $12,12$\\ \cline{2-7}
& $\V_4$ & (V-6) & $[12,5]$ & $\C_3\times \V_4$ & $L_\ba\cap L_{\ba'}=M$ 
& $12,12$\\ \hline\hline
& $\cS_3$ & (V-7) & $[18,3]$ & $\C_3\times \cS_3$ & $L_\ba\cap L_{\ba'}=M$ & $18,6$\\ \cline{2-7}
$\C_3$& \raisebox{-1.6ex}[0cm][0cm]{$\C_3$} & (V-8) & $[9,2]$ & $\C_3\times \C_3$ 
& $L_\ba\cap L_{\ba'}=M$ & $9,9,3,3$\\ \cline{3-7} 
& & (V-9) & $[3,1]$ & $\C_3$ & $L_\ba=L_{\ba'}$ & $3^7,1^3$\\ \cline{1-7}
\end{tabular}
}\vspace*{5mm}
\end{center}
For example, if we take $\ba=(1,-1)$, $\ba'=(-1,1)$ and $M=\mathbb{Q}$ then we have 
$f_\ba^{\cS_4}(X)=(X-1)(X^3+X^2+2X+1)$ and $f_{\ba'}^{\cS_4}(X)=(X+1)(X^3-X^2+1)$. 
We see that $\Spl_\mathbb{Q} f_\ba^{\cS_4}(X)=\Spl_\mathbb{Q} f_{\ba'}^{\cS_4}(X)$ because 
$\mathrm{DT}(\R_{\ba,\ba'})$ is given by $6^2, 3^3, 2, 1$.


{\small 
\vspace*{1mm}
\begin{tabular}{ll}
Akinari HOSHI & Katsuya MIYAKE\\ 
Department of Mathematics & Department of Mathematics\\ 
Rikkyo University & School of Fundamental Science and Engineering\\ 
3--34--1 Nishi-Ikebukuro Toshima-ku & Waseda University\\
Tokyo, 171--8501, Japan  & 3--4--1 Ohkubo Shinjuku-ku\\
E-mail: \texttt{hoshi@rikkyo.ac.jp} & Tokyo, 169--8555, Japan\\
& E-mail: \texttt{miyakek@aoni.waseda.jp}
\end{tabular}
}

\end{document}